\documentclass[reqno]{amsart}
\usepackage[english]{babel}
\usepackage{amscd,amssymb,amsmath,amsfonts,latexsym,amsthm}
\usepackage{inputenc}
\usepackage{graphicx, tikz-cd}

\numberwithin{equation}{section}

\newcommand{\g}{\mathfrak{g}}
\numberwithin{equation}{section}

\newtheorem{thm}{Theorem}[section]
 \newtheorem{cor}[thm]{Corollary}
 \newtheorem{lem}[thm]{Lemma}
 \newtheorem{prop}[thm]{Proposition}
 \newtheorem{defn}[thm]{Definition}
\newtheorem{rem}[thm]{Remark}

\begin{document}
\title[Compatible anti-pre-Lie algebras]{Compatible anti-pre-Lie algebras}
		
\author{Normatov Z.}

\address[Zafar Normatov]{School of Mathematics, Jilin University, Changchun, 130012, China, \newline
Institute of Mathematics, Uzbekistan Academy of Sciences, Univesity Street, 9, Olmazor district, Tashkent, 100174, Uzbekistan}
\email{z.normatov@inbox.ru, \ z.normatov@mathinst.uz}

\begin{abstract} In this paper, we introduce the  notion of compatible anti-pre-Lie algebras and study relationship between them and the related structures such as anti-$\mathcal{O}$-operators, commutative $2$-cocycles on compatible Lie algebras. Moreover, we give the classification of 2-dimensional compatible anti-pre-Lie algebras from the classification of anti-pre-Lie algebras of the same dimension.
\end{abstract}

	
	\maketitle

\section{Introduction}
\

The notion of a pre-Lie algebra was introduced independently by Gerstenhaber \cite{G}, Koszul \cite{K} and Vinberg \cite{V} in the 1960s. Pre-Lie algebras arose from the study of affine manifolds
and affine structures on Lie group, homogeneous convex cones. Pre-Lie algebras have a close relationship with Lie algebras:
a pre-Lie algebra $(A, \ast)$ gives rise to a Lie algebra $(A, [-, -])$ via the commutator bracket, which
is called the subadjacent Lie algebra and denoted by $(\g(A), [-,-])$.

An anti-pre-Lie algebra is a type of algebraic structure closely related to pre-Lie algebras, but with a reversed associativity condition and one extra relation. The notion of this type algebra is introduced by G. Liu G. and  C. Bai in \cite{LB}. In a pre-Lie algebra, the associator satisfies a specific symmetry when two elements are swapped, whereas in an anti-pre-Lie algebra, this symmetry condition is reversed. Similar to pre-Lie algebras, the commutator in an anti-pre-Lie algebra forms a Lie algebra.


Compatible algebraic structures are quite much interest in the sense that linear combinations of multiplications defined on the same two algebraic structures are still of the same kind of algebraic structures. Compatible algebraic structures have
been widely studied in mathematics and mathematical physics. For example, compatible associative algebras were studied in connection with Cartan matrices of affine Dynkin diagrams,
integrable matrix equations, infinitesimal bialgebras and
quiver representations \cite{OS2}--\cite{OS4}. 
Compatible Lie algebras were studied in \cite{OS1} and \cite{GS1}--\cite{GS2} in the contexts of the classical Yang–Baxter equation and principal chiral field, loop algebras over Lie algebras and elliptic theta functions.

The notion of a compatible pre-Lie algebra was introduced by  Wu and Bai in \cite{WB} and studied its properties. Compatible pre-Lie
algebras have a close relation with the classical Yang–Baxter equation in compatible
Lie algebras and as a byproduct, the compatible Lie bialgebras ( introduced in \cite{WB}) fit into the framework to construct non-constant solutions of the classical Yang–Baxter equation given by Golubchik and Sokolov \cite{GS2}.

The classification (up to isomorphism) of algebras of low dimensions from a
certain variety defined by a family of polynomial identities is a classic problem in the theory
of non-associative algebras. There are many results related to the algebraic classification of low-dimensional algebras in many varieties of
associative and non-associative algebras.
For example, the classifications of 2-dimensional algebras in various types are given in \cite{AKM}, \cite{Bai2}, \cite{LB}, \cite{petersson}, \cite{RRB}   and  those of $3$-dimensional algebras are given in   \cite{AANS}, \cite{BOK}, \cite{ccsmv}, \cite{Kobayashi}, \cite{RIB}.

The paper is structured as follows: Firtsly, we introduce the notion of compatible anti-pre-Lie algebras and prove that they are compatible Lie admissible algebras whose negative multiplication operators make representations of commutator compatible Lie algebras in Section 2. In Section 3, we study  relations between anti-$\mathcal{O}$-operators and compatible anti-pre-Lie algebras. In Section 4, we show a close relationship between compatible anti-pre-Lie algebras and commutative $2$-cocycles on compatible Lie algebras which is the former can be induced from the latter in the nondegenerate case. In Section 5, we give the classification of complex 2-dimensional compatible anti-pre-Lie algebras with obtaining 45 non-isomorphic classes of such algebras.

Throughout the paper, all vector spaces and algebras are finite-dimensional and over complex number field
$\mathbb{C}$ unless otherwise stated.

\section{Preliminaries and Notions}

There is an ``anti-structure" for pre-Lie algebras, namely anti-pre-Lie algebras, which is charecterized as  Lie-admissible algebras whose negative left multiplication operators give representations of the commutator Lie algebras.  By the motivation of anti-pre-Lie algebras, in this section, we introduce the notion of compatible anti-pre-Lie algebras as a class of compatible Lie-admissible algebras whose negative left multiplication operators make representations of the commutator compatible Lie algebras.

Recall that  $(A, \ast)$ is called a \textbf{Lie-admissible algebra}, where $A$ is a vector space with a bilinear operation
$\ast : A\otimes A\rightarrow A$, if the commutator operation $[-,-]: A\otimes A\rightarrow A$ makes $(A, [-,-])$ a Lie algebra, where $[x,y]=x\ast y-y\ast x$. In this case, $(A, [-,-])$ is called the \textbf{sub-adjacent Lie algebra} of $(A, \ast)$ denoted by $\g(A)$. 

\begin{defn}
Let $A$ be a vector space with a bilinear operation $\ast$. $(A, \ast)$ is called a pre-Lie algebra if it satisfies the identity
\begin{equation*}
(x\ast y)\ast z - x\ast(y\ast z)=(y\ast x)\ast z - y\ast(x\ast z), \quad \forall x,y,z\in A.
\end{equation*}
\end{defn}

\begin{defn}
Let $A$ be a vector space with a bilinear operation $\ast$. $(A, \ast)$ is called an anti-pre-Lie algebra if it satisfies the identities
\begin{align}\label{anti id1}
&x\ast (y\ast z)-y\ast(x\ast z)=[y,x]\ast z,\\\label{anti id2}
&[x,y]\ast z+[y,z]\ast x+[z,x]\ast y=0, \quad \forall x,y,z\in A,
\end{align}
where the bilinear operation $[-,-]$ is a commutator i.e. $[x,y]=x\ast y-y\ast x$.
\end{defn}

\begin{rem}
Define $\mathcal{L}_{\ast}$ as a left multiplication operator. A pre-Lie algebra $(A, \ast)$ is a Lie-admissible algebra such that $(\mathcal{L}_{\ast}, A)$ is a representation of the sub-adjacent Lie algebra $\g(A)$ while an anti-pre-Lie algebra $(A, \ast)$ is a Lie-admissible algebra such that $(-\mathcal{L}_{\ast}, A)$ is a representation of the sub-adjacent Lie algebra $\g(A)$.
\end{rem}

\begin{defn}
Let $(\g, [-,-]_{1})$ and $(\g, [-,-]_{2})$ be two
Lie algebras. They are called compatible if
for any $k_1, k_2 \in \mathbb{C}$, the following bilinear operation
\begin{equation*}
[x, y] = k_1[x, y]_{1} + k_2[x, y]_{2}, \quad \forall x, y \in \g,    
\end{equation*}
defines a Lie algebra structure on $\g$.
We denote 
two compatible Lie algebras by $(\g, [-,-]_{1}, [-,-]_{2})$ and call
$(\g, [-,-]_{1}, [-,-]_{2})$ a compatible Lie algebra.
\end{defn}

\begin{prop}\cite{WB}\label{compatible lie}
Let $(\g, [-,-]_{1})$ and $(\g, [-,-]_{2})$ be two
Lie algebras. Then $(\g, [-,-]_{1}, [-,-]_{2})$ is a compatible Lie algebra
if and only if for any $x, y, z \in \g$, the following equation holds
\begin{equation}\label{compatible lie algebra}
[[x,y]_{1},z]_{2}+[[y,z]_{1},x]_{2}+[[z,x]_{1},y]_{2}+[[x,y]_{2},z]_{1}+[[y,z]_{2},x]_{1}+[[z,x]_{2},y]_{1}=0.
\end{equation}
\end{prop}

\begin{defn}
A representation of a compatible Lie algebra $(\g, [-,-]_{1}, [-,-]_{2})$ on a vector space $V$ is a pair of linear maps $\rho, \mu : \g \rightarrow gl(V)$ such that for any $k_1, k_2 \in \mathbb{C}$,
$k_1\rho+k_2\mu$ is a representation of the Lie algebra $(\g, k_1[-,-]_{1}+k_2[-,-]_{2})$. We denote it by $(\rho, \mu, V)$.
\end{defn}

\begin{prop}\cite{WB}\label{representation}
Let $(\g, [-,-]_{1}, [-,-]_{2})$ be a compatible Lie
algebra. Let $V$ be a vector space and $\rho, \mu : \g \rightarrow gl(V)$ be
a pair of linear maps. Then the following conditions are
equivalent:
\begin{itemize}
    \item[(i)] $(\rho, \mu, V)$ is a representation of $(\g, [-,-]_{1}, [-,-]_{2})$ on the
vector space $V$.
    \item[(ii)] For any $x, y \in \g$, the following three equations hold
\begin{align}\label{rho}
&1. \quad \rho([x, y]_{1}) = \rho(x)\rho(y) - \rho(y)\rho(x),\\\label{mu}
&2. \quad \mu([x, y]_{2}) = \mu(x)\mu(y) - \mu(y)\mu(x),\\\label{rho-mu}
&3. \quad \rho([x, y]_{2}) + \mu([x, y]_{1}) = \rho(x)\mu(y)
- \rho(y)\mu(x) + \mu(x)\rho(y) - \mu(y)\rho(x).
\end{align}
\end{itemize}
\end{prop}

\begin{defn}
Let $(A, \circ)$ and $(A, \ast)$ be two anti-pre-Lie algebras. They are called \textbf{\textit{compatible}} if for any $k_1, k_2 \in \mathbb{C}$, the following bilinear operation
\begin{equation}\label{compatible id}
x \star y = k_1x \circ y + k_2x \ast y , \quad \forall x, y \in A,
\end{equation}
defines an anti-pre-Lie algebra structure on $A$.
We denote
the two compatible anti-pre-Lie algebras by $(A, \circ, \ast)$ and call
$(A, \circ, \ast)$ a \textbf{\textit{compatible anti-pre-Lie algebra}}.
\end{defn}

\begin{prop}\label{interchange}
Let $(A, \circ)$ and $(A, \ast)$ be two anti-pre-Lie algebras. Then $(A, \circ, \ast)$ is a compatible anti-pre-Lie algebra if and only if for any $x,y,z\in A$, the following equalities hold:
\begin{align}
\label{compatible anti-pre id1}
&1. \quad x\circ(y\ast z)+x\ast (y\circ z)-y\circ(x\ast z)-y\ast(x\circ z)=[y,x]_2\circ z+[y,x]_1\ast z,\\ \label{compatible anti-pre id2}
&2. \quad [x,y]_2\circ z+[x,y]_1\ast z+[y,z]_2\circ x+[y,z]_1\ast x+[z,x]_2\circ y+[z,x]_1\ast y=0,
\end{align}
where 
\begin{equation}\label{commutators}
[x,y]_1=x\circ y-y\circ x \quad  \text{and} \quad  [x,y]_2=x\ast y-y\ast x. 
\end{equation}
\end{prop}

\begin{proof}
Let $x,y,z \in A$. By \eqref{compatible id} and \eqref{anti id1} in the definition of anti-pre-Lie algebras, we have 
\[
x\star(y\star z)-y\star(x\star z)=(x\star y-y\star x)\star z.
\]
Then, we obtain
\begin{align*}\nonumber
LHS=&k_1^2(x\circ(y\circ z)-y\circ(x\circ z))+k_2^2(x\ast(y\ast z)-y\ast(x\ast z))\\
&+k_1k_2(x\circ(y\ast z)+x\ast(y\circ z)-y\circ(x\ast z)-y\ast(x\circ z)),\\\nonumber
RHS=&k_1^2((y\circ x)\circ z-(x\circ y)\circ z)+k_2^2((y\ast x)\ast z-(x\ast y)\ast z)\\
&+k_1k_2((y\ast x)\circ z+(y\circ x)\ast z-(x\ast y)\circ z-(x\circ y)\ast z).
\end{align*}
By equalizing we get \eqref{compatible anti-pre id1}. Similarly, by considering \eqref{anti id2} one can take \eqref{compatible anti-pre id2}.

Let \eqref{compatible anti-pre id1} and \eqref{compatible anti-pre id2} hold. Then
repeating the same calculation in the other direction one can show that $(A,\circ, \ast)$ is a compatible anti-pre-Lie algebra.
\end{proof}

\begin{rem}
Let $(A,\circ, \ast)$ be a compatible anti-pre-Lie algebra and  $x,y,z\in A$. 
Then we can write
\begin{equation}\label{left to right}
\begin{split}
&x\circ [y,z]_2+x\ast [y,z]_1+y\circ [z,x]_2+y\ast [z,x]_1+z\circ [x,y]_2+z\ast [x,y]_1\\
\stackrel{\eqref{commutators}}{=}&x\circ (y\ast z)-x\circ (z\ast y)+x\ast (y\circ z)-x\ast (z\circ y)+y\circ (z\ast x)-y\circ (x\ast z)\\
&+y\ast (z\circ x)-y\ast (x\circ z)+z\circ (x\ast y)-z\circ (y\ast x)+z\ast (x\circ y)-z\ast (y\circ x)\\
\stackrel{\eqref{commutators}}{=}&[y,x]_2\circ z+[y,x]_1\ast z+[x,z]_2\circ y+[x,z]_1\ast y+[z,y]_2\circ x+[z,y]_1\ast x.
\end{split}
\end{equation}
Thus, \eqref{compatible anti-pre id2} holds if and only if following equality holds
\begin{equation}\label{2-half jacobian}
x\circ [y,z]_2+x\ast [y,z]_1+y\circ [z,x]_2+y\ast [z,x]_1+z\circ [x,y]_2+z\ast [x,y]_1=0.
\end{equation}
Thus we can interchange \eqref{compatible anti-pre id2} with \eqref{2-half jacobian} in Proposition \ref{interchange}.
\end{rem}
Let $(A,\circ)$ be an anti-pre-Lie algebra and $L_{\circ}:A\rightarrow gl(A)$ be a linear map defined by $L_{\circ}(x)y=x\circ y$, for any $x,y\in A$.

\begin{prop}
Let $(A, \circ, \ast)$ be a compatible anti-pre-Lie algebra. Then  $(-L_{\circ}, -L_{\ast}, A)$ is a representation of the sub-adjacent  compatible Lie algebra $\g(A)$.
\end{prop}
\begin{proof}
Since $(A, \circ)$ and $(A, \ast)$ are anti-pre-Lie algebras, by \eqref{anti id1} one can obtain \eqref{rho} and \eqref{mu}.

By \eqref{compatible anti-pre id1} we have 
\[
\begin{split}
-L_{\circ}([x,y]_2)(z)-L_{\ast}([x,y]_1)(z)&=[y,x]_2\circ z+[y,x]_1\ast z\\
&=x\circ (y\ast z)-y\circ (x\ast z)+x\ast (y\circ z)-y\ast (x\circ z)\\
&=L_{\circ}(x)L_{\ast}(y)(z)-L_{\circ}(y)L_{\ast}(x)(z)+L_{\ast}(x)L_{\circ}(y)(z)-L_{\ast}(y)L_{\circ}(x)(z).
\end{split}
\]
The last equality shows that it satisfies \eqref{rho-mu}. Thus $(-L_{\circ}, -L_{\ast}, A)$ is a representation by Proposition~\ref{representation}.
\end{proof}

Let $(A, \circ,\ast)$ be a compatible algebra. We call $(A, \circ,\ast)$ a \textbf{compatible Lie-admissible algebra} if the commutators $[-,-]_{1}, [-,-]_{2}:A\otimes A\rightarrow A$ make $(A, [-,-]_{1}, [-,-]_{2})$ a compatible Lie algebra where $[-,-]_1$ and $[-,-]_2$ are defined by \eqref{commutators}.

\begin{prop}
Let $A$ be a vector space with two bilinear operations $\circ, \ast: A\otimes A\rightarrow A$ such that $(A, \circ)$ and $(A, \ast)$ are anti-pre-Lie algebras. Then $(A, \circ, \ast)$ is a compatible anti-pre-Lie algebra if and only if  $(A, \circ, \ast)$ is a compatible Lie admissible algebra with a representation $(-\mathcal{L}_{\circ},-\mathcal{L}_{\ast}, A)$.
\end{prop}
\begin{proof}
Suppose $(A, \circ, \ast)$ is a compatible anti-pre-Lie algebra. By (Proposition 2.2, \cite{LB}) $(A, \circ)$ and $(A,  \ast)$ are Lie-admissible algebras implying that $(A, [-,-]_{1})$ and $(A, [-,-]_{2})$ are Lie algebras. 

The difference of \eqref{compatible anti-pre id2} and \eqref{2-half jacobian} derives
\[
[x,[y,z]_2]_1+[x,[y,z]_1]_2+[y,[z,x]_2]_1+[y,[z,x]_1]_2+[z,[x,y]_2]_1+[z,[x,y]_1]_2=0,
\]
which is equivalent to \eqref{compatible lie algebra}. Hence, by Proposition \ref{compatible lie} $(A, [-,-]_1, [-,-]_2)$ is a compatible Lie algebra.

Conversely, let $(A, \circ, \ast)$ be a compatible Lie admissible algebra and $(-\mathcal{L}_{\circ},-\mathcal{L}_{\ast}, A)$ be its representation. According to (Proposition 2.2 \cite{LB}) $(A, \circ)$ and $(A, \ast)$ are anti-pre-Lie algebras whose representations are $(-\mathcal{L}_{\circ}, A)$ and $(-\mathcal{L}_{\ast}, A)$, respectively. By \eqref{rho-mu} for $\forall x,y,z\in A$ we have
\[
-\mathcal{L}_{\circ}([x,y]_2)z-\mathcal{L}_{\ast}([x,y]_1)z=\mathcal{L}_{\circ}(x)\mathcal{L}_{\ast}(y)z-\mathcal{L}_{\circ}(y)\mathcal{L}_{\ast}(x)z+\mathcal{L}_{\ast}(x)\mathcal{L}_{\circ}(y)z-\mathcal{L}_{\ast}(y)\mathcal{L}_{\circ}(x)z.
\]
Thus we have
\[
[y,x]_2\circ z+[y,x]_1\ast z=x\circ(y\ast z)-y\circ(x\ast z)+x\ast(y\circ z)-y\ast(x\circ z)
\]
which implies that \eqref{compatible anti-pre id1} holds.

Since $(A,[-,-]_1,[-,-])$ is a compatible Lie algebra, \eqref{compatible lie algebra} holds. By expanding \eqref{compatible lie algebra} one can show that \eqref{left to right} also holds. Hence, \eqref{compatible anti-pre id2} holds and this proves that $(A, \circ, \ast)$ is a compatible anti-pre-Lie algebra.
\end{proof}

\begin{defn}
A compatible associative algebra is a triple $(A, \cdot_1, \cdot_2)$ in which $(A, \cdot_1)$ and $(A, \cdot_2)$
are both associative algebras satisfying the following compatibility condition:
\[
(x \cdot_1 y) \cdot_2 z + (x \cdot_2 y) \cdot_1 z = x \cdot_1 (y \cdot_2 z) + x \cdot_2 (y \cdot_1 z), 
\]
for $x, y, z \in A$.
\end{defn}

\begin{prop}
Let $A$ be a vector space with two bilinear operations $\circ, \ast: A\otimes A\rightarrow A$. Suppose both operations are commutative i.e.
\[
x\circ y=y\circ x, \quad x\ast y=y\ast x, \quad \forall \ x,y\in A.
\]
Then $(A,\circ,\ast)$ is a compatible anti-pre-Lie algebra if and only if $(A,\circ,\ast)$ is a compatible associative algebra.
\end{prop}
\begin{proof}
It is obvious that \eqref{compatible anti-pre id2} and \eqref{anti id2} for both operations are hold. Furthermore, we have
\begin{align*}
x\circ (y\circ z)-y\circ(x\circ z)-[y,x]\circ z=(y\circ z)\circ x-y\circ (z\circ x),\\
x\ast (y\ast z)-y\ast(x\ast z)-[y,x]\ast z=(y\ast z)\ast x-y\ast (z\ast x),
\end{align*}
and
\begin{align}\nonumber
&x\circ (y\ast z)+x\ast (y\circ z)-y\circ (x\ast z)-y\ast(x\circ z)-[y,x]_2\circ z-[y,x]_1\ast z\\\label{compatible asso}
=&(x\circ y)\ast z+(x\ast y)\circ z-x\circ (y\ast z)-x\ast (y\circ z).
\end{align}
Then, \eqref{anti id1} holds for both operations if and only if $(A, \circ)$ and $(A, \ast)$ are associative. \eqref{compatible anti-pre id1} holds if and only if the right hand side of \eqref{compatible asso} equals zero. Hence the conclusion follows.
\end{proof}

\section{Compatible anti-pre-Lie algebras and anti-$\mathcal{O}$-operators on compatible Lie algebras.}

In this section we study  anti-$\mathcal{O}$-operators to interpret compatible anti-pre-Lie algebras. There is a structure of a compatible anti-pre-Lie algebra on
the representation space inducing from a strong anti-$\mathcal{O}$-operator on a compatible Lie algebra and in particular,
the existence of an invertible anti-$\mathcal{O}$-operator gives an equivalent condition for a compatible Lie algebra having a structure of a  compatible anti-pre-Lie algebra.

\begin{defn}\cite{DGC}
Let $(\g,[-,-])$ be a Lie algebra and $(\rho,V)$ be a
representation. A linear map $T:V\rightarrow \g$ is called an
\textbf{anti-$\mathcal O$-operator} 
associated to $(\rho,V)$ if $T$ satisfies
\begin{equation*}
[T(u),T(v)]=T(\rho(T(v))u-\rho(T(u))v), \forall u,v\in V.
\end{equation*}
Moreover,  anti-$\mathcal{O}$-operator $T$ is called strong if $T$ satisfies
\begin{equation*}
\rho([T(u),T(v)])w+\rho([T(v),T(w)])u+\rho([T(w),T(u)])v=0, \quad \forall \ u,v,w\in V.
\end{equation*}
\end{defn}

\begin{defn}
Let $(\g,[-,-]_{1}, [-,-]_{2})$ be a compatible Lie algebra and $(\rho, \mu, V)$ be a
representation. A linear map $T:V\rightarrow \g$ is called an
\textbf{anti-$\mathcal O$-operator} 
associated to $(\rho, \mu, V)$ if for any $k_1,k_2\in\mathbb{C}$, $T$ is an anti-$\mathcal{O}$-operator of the Lie algebra $(\g, k_1[-,-]_{1}+k_2[-,-]_{2})$ associated to the representation $k_1\rho+k_2\mu$.
\end{defn}

An anti-$\mathcal{O}$-operator of $(\g, [-, -])$ associated to the adjoint representation $(\mathrm{ad}, \g)$ which is defined by $\mathrm{ad}(x)(y)=[x,y]$ for all $x,y\in \g$
is called an anti-Rota-Baxter operator, that is, $R : \g \rightarrow \g$ is a linear map satisfying
\begin{equation*}
[R(x), R(y)] = R([R(y), x] + [y, R(x)]), \quad \forall \ x, y \in \g.  
\end{equation*}

An anti-Rota-Baxter operator $R$ is called strong if $R$ satisfies
\begin{equation}\label{strong anti rota}
[[R(x), R(y)], z] + [[R(y), R(z)], x] + [[R(z), R(x)], y] = 0, \quad \forall \ x, y, z \in \g.   
\end{equation}

\begin{prop}\label{t is strong}
Let $(\g, [-,-]_{1}, [-,-]_{2})$ be a compatible Lie
algebra and $(\rho, \mu, V) : \g \rightarrow gl(V)$ be a representation.
Let $T : V \rightarrow \g$ be an anti-$\mathcal{O}$-operator associated to $(\rho, \mu, V)$. Define
\begin{equation}\label{rho-mu to compatible anti-pre}
u \circ v = -\rho(T(u))v, \quad 
u \ast v = -\mu(T(u))v, \quad 
\forall \  u, v \in V.
\end{equation}
Then $(V, \circ, \ast)$ satisfies \eqref{compatible anti-pre id1}. Moreover, $(V, \circ, \ast)$ is a compatible Lie-admissible algebra such that $(V, \circ, \ast)$ is a compatible anti-pre-Lie algebra if and only if $T$ is strong.

Furthermore, $T$ is a homomorphism of compatible Lie algebras from the sub-adjacent Lie algebra $\g(V)$ to $(\g, [-,-]_{1}, [-,-]_{2})$. Moreover, there is an induced compatible anti-pre-Lie algebra structure on $T(V)=\{T(u)\ | \ u\in V\}$ given by
\begin{equation}\label{tu,tv}
T(u)\circ T(v)=T(u\circ v),  \quad T(u)\ast T(v)=T(u\ast v),  \quad \forall \ u,v\in V,
\end{equation}
and $T$ is a homomorphism of compatible anti-pre-Lie algebras.
\end{prop}
\begin{proof}
By Proposition 2.14 in \cite{LB} $T$ is a strong anti-$\mathcal{O}$-operator of $(\g, [-,-]_{1})$ and $(\g, [-,-]_{2})$ associated to $(\rho,V)$ and $(\mu, V)$, respectively, if and only if $(V, \circ)$ and $(V, \ast)$ are Lie admissible.

Let $u,v,w\in V$. Then
\[
\begin{split}
[u,v]_{2}\circ w+[u,v]_{1}\ast w=&(\mu(T(v))u-\mu(T(u))v)\circ w+(\rho(T(v))u-\rho(T(u))v)\ast w\\
=&-\rho(T(\mu(T(v))u))w+\rho(T(\mu(T(u))v))w-\mu(T(\rho(T(v))u))w+\mu(T(\rho(T(u))v))w\\
=&\rho([T(v),T(u)]_{2})w+\mu([T(v),T(u)]_{1})w\\
=&\rho(T(v))\mu(T(u))w-\rho(T(u))\mu(T(v))w+\mu(T(v))\rho(T(u))w-\mu(T(u))\rho(T(v))w\\
=&v\circ (u\ast w)-u\circ (v\ast w)+v\ast (u\circ w)-u\ast(v\circ w).
\end{split}
\]
Thus \eqref{compatible anti-pre id1} holds on $(V, \circ, \ast)$. Moreover, $(V, \circ, \ast)$ is a compatible Lie-admissible algebra if and only if \eqref{compatible anti-pre id2} holds on $(V, \circ, \ast)$, that is,
\[
\begin{split}
&[u,v]_2\circ w+[u,v]_1\ast w+[v,w]_2\circ u+[v,w]_1\ast u+[w,u]_2\circ v+[w,u]_1\ast v\\
=&\rho([T(v),T(u)]_{2})w+\mu([T(v),T(u)]_{1})w+\rho([T(w),T(v)]_{2})u\\
&+\mu([T(w),T(v)]_{1})u+\rho([T(u),T(w)]_{2})v+\mu([T(u),T(w)]_{1})v\\
=&\frac{1}{k_1k_2}\Big(k_1^2\rho([T(u),T(v)]_{1})w+k_1^2\rho([T(v),T(w)]_{1})u+k_1^2\rho([T(w),T(u)]_{1})v\\
&+k_1k_2\rho([T(u),T(v)]_{2})w+k_1k_2\rho([T(v),T(w)]_{2})u+k_1k_2\rho([T(w),T(u)]_{2})v\\
&+k_1k_2\mu([T(u),T(v)]_{1})w+k_1k_2\mu([T(v),T(w)]_{1})u+k_1k_2\mu([T(w),T(u)]_{1})v\\
&+k_2^2\mu([T(u),T(v)]_{2})w+k_2^2\mu([T(v),T(w)]_{2})u+k_2^2\mu([T(w),T(u)]_{2})v\Big)\\
=&\frac{1}{k_1k_2}\Big((k_1\rho+k_2\mu)([T(u),T(v)]_{k_11+k_22})w+(k_1\rho+k_2\mu)([T(v),T(w)]_{k_11+k_22})u\\
&+(k_1\rho+k_2\mu)([T(w),T(u)]_{k_11+k_22})v\Big)\\
=&0.
\end{split}
\]
Note that this is also true for $k_1k_2=0$. The other parts are straightforward.
\end{proof}

\begin{cor}
Let $(\g, [-, -]_{1}, [-, -]_{2})$ be a compatible Lie algebra and $R : \g \rightarrow \g$ be a strong anti-Rota-Baxter
operator. Then
\begin{equation}\label{anti-rota to anti-pre}
x \circ y = -[R(x), y]_{1}, \quad x\ast y=-[R(x), y]_{2}, \quad \forall \ x, y \in \g    
\end{equation}
defines a compatible anti-pre-Lie algebra $(\g, \circ, \ast)$. Conversely, if $R:\g\rightarrow \g$ is a linear transformation on a compatible Lie algebra $(\g, [-,-]_{1}, [-,-]_{2})$ such that \eqref{anti-rota to anti-pre} defines a compatible anti-pre-Lie algebra, then $R$ satisfies \eqref{strong anti rota} and the following equation:
\[
[[R(x),R(y)]_{\g}+R([x,R(y)]_{\g}+[R(x),y]_{\g}),z]_{\g}=0,\quad  \forall \ x,y,z\in \g
\]
\end{cor}
\begin{proof}
By setting $\rho=\mathrm{ad}_1, \mu=\mathrm{ad}_2$, the first part follows from above Proposition. The second part comes from the definitions and Corollary 2.15 in \cite{LB}.
\end{proof}


Now we consider  invertible anti-$\mathcal{O}$-operators.

\begin{prop}\label{invertible to strong}
Let $(\g,[-,-]_{1}, [-,-]_{2})$ be a compatible Lie algebra and $(\rho, \mu, \g)$ be its representation. Then an invertible anti-$\mathcal{O}$-operator associated to $(\rho, \mu, \g)$ is strong.
\end{prop}
\begin{proof}
Let $T : V \rightarrow \g$ be an anti-$\mathcal{O}$-operator of a compatible Lie algebra $(\g, [-, -]_{1},[-,-]_{2})$ associated to a representation $(\rho, \mu, V )$. Define  bilinear operations $\circ, \ast : V \otimes V \rightarrow V$ by \eqref{rho-mu to compatible anti-pre}. Then by choosing $k_1=1, k_2=0$ and $k_1=0, k_2=1$ we know $T$ is a strong anti-$\mathcal{O}$-operator for $(V, [-,-]_{1})$ and $(V, [-,-]_{2})$ such that $(V, \circ)$ and $(V, \ast)$ are anti-pre-Lie algebras  (see, Proposition 2.17,  \cite{LB}). For all $u,v\in V$ we have
\[
\begin{split}
[u,v]_{\g}=&k_1[u,v]_{1}+k_2[u,v]_{2}\\
=&k_1(\rho(T(v))u-\rho(T(u))v)+k_2(\mu(T(v))u-\mu(T(u))v)\\
=&(k_1\rho+k_2\mu)(T(v))u-(k_1\rho+k_2\mu)(T(u))v\\
=&T^{-1}([T(u),T(v)]_{\g}).
\end{split}
\]
Therefore, for all $u,v,w\in V$, we have
\[
\begin{split}
[[u,v]_{\g},w]_{\g}=&[T^{-1}([T(u),T(v)]_{\g}),w]_{\g}\\
=&T^{-1}([[T(u),T(v)]_{\g},T(w)]_{\g})
\end{split}
\]
This implies $(V,[-,-]_{\g })$ is a Lie algebra as well as $(V,[-,-]_{1}, [-,-]_{2})$ is a compatible Lie algebra such that $(V,\circ, \ast)$ is  compatible Lie admissible. Then $T$ is strong due to Proposition \ref{t is strong}.
\end{proof}

\begin{cor}
Let $(\g, [-,-]_{1}, [-,-]_{2})$ be a compatible Lie algebra. Then there is a structure of a  compatible anti-pre-Lie algebra on $\g$ if and only if there exists an invertible anti-$\mathcal{O}$-operator of $(\g, [-,-]_{1}, [-,-]_{2})$.
\end{cor}
\begin{proof}
Let $(\g, \circ, \ast)$ be a structure of a  compatible anti-pre-Lie algebra on $(\g, [-,-]_{1}, [-,-]_{2})$. Then
\[
\begin{split}
[x,y]_{\g}=&k_1[x,y]_{1}+k_2[x,y]_{2}\\
=&k_1(x\circ y-y\circ x)+k_2(x\ast y-y\ast x)\\
=&k_1(-\mathcal{L}_{\circ}(y)-(-\mathcal{L}_{\circ}(x))y)+k_2(-\mathcal{L}_{\ast}(y)-(-\mathcal{L}_{\ast}(x))y)
\end{split}
\]
Thus $T=\mathrm{id}$ is an invertible anti-$\mathcal{O}$-operator of $(\g, [-,-]_{1}, [-,-]_{2})$ associated to $(-\mathcal{L}_{\circ}, -\mathcal{L}_{\ast},\g)$.

Conversely, suppose that $T:V\rightarrow \g$ is an invertible anti-$\mathcal{O}$-operator of $(\g, [-,-]_{1}, [-,-]_{2})$ associated to $(\rho, \mu, V)$. Then by Proposition \ref{invertible to strong}, $T$ is strong. Also, there is a structure of a compatible anti-pre-Lie algebra on the underlying vector space of $\g$ given by \eqref{tu,tv} by Proposition \ref{t is strong}. Explicitly, for any $x,y\in\g$, there exist $u,v\in V$ such that $x=T(u), y=T(v)$. Hence, 
\[
\begin{split}
&x\circ_{\g} y=T(u)\circ_{\g}T(v)=T(u\circ v)=-T(\rho(x)T ^{-1}(y)), \quad \forall \ x,y\in\g,\\
&x\ast_{\g} y=T(u)\ast_{\g}T(v)=T(u\ast v)=-T(\mu(x)T ^{-1}(y)), \quad \forall \ x,y\in\g.
\end{split}
\]
Furthermore, we have
\[
\begin{split}
[x,y]_{\g}=&[T(u),T(v)]\\
=&T((k_1\rho+k_2\mu)(T(v))u-(k_1\rho+k_2\mu)(T(u))v)\\
=&k_1T(\rho(y)T ^{-1}(x)-\rho(x)T ^{-1}(y))+k_2T(\mu(y)T ^{-1}(x)-\mu(x)T ^{-1}(y))\\
=&k_1(x\circ_{\g}y-y\circ_{\g}x)+k_2(x\ast_{\g}y-y\ast_{\g}x).
\end{split}
\]
Thus $(\g, \circ, \ast)$ is a compatible anti-pre-Lie algebra whose sub-adjacent Lie algebra is $(\g, [-,-]_{1}, [-,-]_{2})$.
\end{proof}

\section{Compatible anti-pre-Lie algebras and commmutative $2$-cocycles on compatible Lie algebras}

Let us recall the notion of commutative $2$-cocycles on Lie algebras.
\begin{defn}\cite{DZ}
A commutative $2$-cocycle $\mathcal{B}$ on a Lie algebra $(\g, [-,-])$ is a symmetric
bilinear form such that
\[
\mathcal{B}([x,y],z)+\mathcal{B}([y,z],x)+\mathcal{B}([z,x],y)=0, \quad \forall \ x,y,z\in\g.
\]
\end{defn}

\begin{prop}\cite{WB}
Let $(\rho, \mu, V)$ be a representation of a
compatible Lie algebra $(\g, [-,-]_{1}, [-,-]_{2})$. Then $(\rho^{*}, \mu^{*}, V^{*})$
is a representation of $(\g, [-,-]_{1}, [-,-]_{2})$, which is called the
dual representation of $(\rho, \mu, V)$, where $\rho^{*}$ and $\mu^{*}$ are given by 
\[
\begin{split}
&\langle\rho^{*}(x)a^{*},b\rangle=-\langle a^{*},\rho(x) b\rangle, \quad \forall \ x\in\g, \ a^{*}\in V^{*}, \ b\in V,\\
&\langle\mu^{*}(x)a^{*},b\rangle=-\langle a^{*},\mu(x) b\rangle, \quad \forall \ x\in\g, \ a^{*}\in V^{*}, \ b\in V.
\end{split}
\]
In particular, $(\mathrm{ad}^{*}_{1}, \mathrm{ad}^{*}_{2}, \g^{*})$ is the dual representation of the representation $(\mathrm{ad}_{1}, \mathrm{ad}_{2}, \g)$.
\end{prop}

\begin{thm}\label{thm:commutative $2$-cocycles and anti-pre-Lie
algebras}
Let $\mathcal{B}$ be a nondegenerate commutative $2$-cocycle on a compatible Lie algebra $(\g, [-,-]_{1}, [-,-]_{2})$. Then there exists a structure of a compatible anti-pre-Lie algebra on $(\g, [-,-]_{1}, [-,-]_{2})$ given by
\[
\mathcal{B}(x\circ y, z)=\mathcal{B}(y,[x,z]_{1}), \quad \mathcal{B}(x\ast y,z)=\mathcal{B}(y,[x,z]_{2}), \quad \forall \ x,y,z\in\g.
\]
\end{thm}
\begin{proof}
Define a linear map $T:\g \rightarrow \g^{*}$ by
\[
\langle T(x),y\rangle=\mathcal{B}(x,y), \quad \forall \ x,y\in\g.
\]
Then $T$ is invertible by the nondegeneracy of $\mathcal{B}$. Moreover, for any $a^{*}, b^{*}\in \g^{*}$, there exist $x,y\in \g$ such that $a^{*}=T(x), b^{*}=T(y)$. Then for any $z\in \g$, we have
\[
\begin{split}
\mathcal{B}([x,y]_{\g},z)=&\langle T([x,y]_{\g}),z\rangle \\
=&\langle T([T^{-1}(a^{*}),T^{-1}(b^{*})]_{\g}),z\rangle
\end{split}
\]
\[
\begin{split}
\mathcal{B}([y,z]_{\g},x)=&\mathcal{B}(x,[y,z]_{\g})\\
=&\langle T(x),k_1[y,z]_{\g1}+k_2[y,z]_{2}\rangle \\
=&k_1\langle T(x),\mathrm{ad}_{1}(y)(z)\rangle+k_2\langle T(x),\mathrm{ad}_{2}(y)(z)\rangle\\
=&-k_1\langle \mathrm{ad}^{*}_{1}(y)T(x), z\rangle-k_2\langle \mathrm{ad}^{*}_{2}(y)T(x), z\rangle,
\end{split}
\]
\[
\begin{split}
\mathcal{B}([z,x]_{\g},y)=&\mathcal{B}(y,[z,x]_{\g})\\
=&\langle T(y),k_1[z,x]_{\g1}+k_2[z,x]_{2}\rangle \\
=&-k_1\langle T(y),\mathrm{ad}_{1}(x)(z)\rangle-k_2\langle T(y),\mathrm{ad}_{2}(x)(z)\rangle\\
=&k_1\langle \mathrm{ad}^{*}_{1}(x)T(y), z\rangle+k_2\langle \mathrm{ad}^{*}_{2}(x)T(y), z\rangle.
\end{split}
\]
Since $\mathcal{B}$ is a commutative $2$-cocycle, we have
\[
\begin{split}
[T^{-1}(a^{*}),T^{-1}(b^{*})]_{\g}=&T^{-1}(k_1 \mathrm{ad}^{*}_{1}(x)T(y)+k_2\mathrm{ad}^{*}_{2}(x)T(y)-k_1 \mathrm{ad}^{*}_{1}(y)T(x)-k_2\mathrm{ad}^{*}_{2}(y)T(x))\\
=&T^{-1}((k_1 \mathrm{ad}^{*}_{1}+k_2\mathrm{ad}^{*}_{2})(x)T(y)-(k_1 \mathrm{ad}^{*}_{1}+k_2\mathrm{ad}^{*}_{2})(y)T(x))
\end{split}
\]
Thus $T^{-1}$ is an anti-$\mathcal{O}$-operator of $(\g, [-, -]_{1}, [-,-]_{2})$ associated to $(\mathrm{ad}_{1}^{*}, \mathrm{ad}_{2}^{*}, \g^{*})$. By Corollary 2.18, there is a structure of a compatible anti-pre-Lie algebra on $(\g, [-, -]_{1}, [-,-]_{2})$ given by
\[
x\circ y=-T^{-1}(\mathrm{ad}_{1}^{*}(x)T(y)), \quad  x\ast y=-T^{-1}(\mathrm{ad}_{2}^{*}(x)T(y)), \quad \forall \ x,y\in \g,
\]
which gives
\[
\mathcal{B}(x\circ y, z) =\langle T(x \circ y), z\rangle = -\langle \mathrm{ad}_{1}^{*}(x)T(y), z\rangle = \langle T(y), [x, z]\rangle = \mathcal{B}(y, [x, z]_{1}), \quad \forall x, y, z \in \g.
\]
\[
\mathcal{B}(x\ast y, z) =\langle T(x \ast y), z\rangle = -\langle \mathrm{ad}_{2}^{*}(x)T(y), z\rangle = \langle T(y), [x, z]\rangle = \mathcal{B}(y, [x, z]_{2}), \quad \forall x, y, z \in \g.
\]
Furthermore, we have
\[
\mathcal{B}(k_1x\circ y+k_2x\ast y, z) =\mathcal{B}(y, k_1[x, z]_{1}+k_2[x,z]_{2}), \quad \forall x, y, z \in \g.
\]
Hence, the conclusion holds.
\end{proof}

\begin{defn}
A bilinear form $\mathcal{B}$ on an anti-pre-Lie algebra $(A,\circ)$ is called invariant if 
\begin{equation}\label{invariant $2$-cocycle}
\mathcal{B}(x\circ y,z)=\mathcal{B}(y,[x,z]), \quad \forall \ x,y,z\in\g  
\end{equation}
holds.
\end{defn}
Note that according to the definition of a compatible anti-pre-Lie algebra, if a bilinear form $\mathcal{B}$ is invariant on $(A,\circ,\ast)$, then it is also invariant on $(A,\circ)$ and $(A, \ast)$.
\begin{cor}\label{cor:corres}
Any symmetric invariant bilinear form on a compatible anti-pre-Lie algebra
$(A,\circ, \ast)$ is a commutative $2$-cocycle on the sub-adjacent compatible Lie
algebra $(\g(A),[-,-]_{1}, [-,-]_{2})$. Conversely, a nondegenerate
commutative $2$-cocycle on a compatible Lie algebra $(\g,[-,-]_{1}, [-,-]_{2})$ is
invariant on a compatible anti-pre-Lie algebra given by
\begin{equation}\label{eq:thm:commutative $2$-cocycles and anti-pre-Lie
algebras}
\mathcal{B}(x\circ y,z)=\mathcal{B}(y,[x,z]_{1}), \quad \mathcal{B}(x\ast y,z)=\mathcal{B}(y,[x,z]_{2}) \quad \forall \ x,y,z\in\g    
\end{equation}
\end{cor}
\begin{proof}
For the first half part, by \eqref{invariant $2$-cocycle}, we have
\begin{equation*}
\begin{split}
\mathcal{B}([x,y]_{\g},z)=&\mathcal{B}(k_1[x,y]_{1}+k_2[x,y]_{2},z)\\
=&k_1\mathcal{B}([x,y]_{1})+k_2\mathcal{B}([x,y]_{2},z)\\
=&k_1(\mathcal{B}(x\circ y,z)-\mathcal{B}(y\circ x,z))+k_2(\mathcal{B}(x\ast y,z)-\mathcal{B}(y\ast x,z))\\
=&k_1(\mathcal{B}(y,[x,z]_{1})-\mathcal{B}(x,[y,z]_{1}))+k_2(\mathcal{B}(y,[x,z]_{2})-\mathcal{B}(x,[y,z]_{2}))\\
=&\mathcal{B}(y,k_1[x,z]_{1}+k_2[x,z]_{2})-\mathcal{B}(x,k_1[y,z]_{1}+[y,z]_{2})\\
=&\mathcal{B}(y,[x,z]_{\g})-\mathcal{B}(x,[y,z]_{\g})
\end{split}
\end{equation*}

Thus $\mathcal{B}$ is a commutative $2$-cocycle on the sub-adjacent compatible
Lie algebra $(\g(A),[-,-]_{1}, [-,-]_{2})$. The second half part follows
from Theorem~\ref{thm:commutative $2$-cocycles and anti-pre-Lie
algebras}.
\end{proof}

Recall that two representations $(\rho_{1},V_{1})$ and
$(\rho_{2},V_{2})$ of a Lie algebra $(\g,[-,-])$ are called
\textbf{equivalent} if there is a linear isomorphism
$\varphi:V_{1}\rightarrow V_{2}$ such that
$\varphi(\rho_{1}(x)v)=\rho_{2}(x)\varphi(v), \forall x\in \g, v\in V_{1}$.

\begin{prop}\label{pro:equivalence and invariance}
Let $(A,\circ, \ast)$ be a compatible anti-pre-Lie algebra. Then there is a
nondegenerate invariant bilinear form on $(A,\circ, \ast)$ if and only
if $(-\mathcal{L}_{\circ}, -\mathcal{L}_{\ast}, A)$ and $(\mathrm{ad}_{1}^{*}, \mathrm{ad}_{2}^{*}, A^{*})$ are
equivalent as representations of the sub-adjacent compatible Lie algebra
$(\g(A),[-,-]_{1}, [-,-]_{2})$.
\end{prop}

\begin{proof}
Suppose  $\varphi:A\rightarrow A^{*}$ is a linear isomorphism
satisfying
$$\varphi (-(k_1\mathcal L_\circ(x)+k_2\mathcal L_\circ(x))y)=(k_1{\rm
ad}_{1}^*(x)+k_2{\rm
ad}_{2}^*(x))\varphi(y),\;\;\forall x,y\in A.$$ Define a nondegenerate
bilinear form $\mathcal{B}$ on $A$ by
\begin{equation}\label{eq:pro:equivalence and invariance}
\mathcal{B}(x,y)=\langle \varphi(x),y\rangle,\;\; \forall x,y\in
A.
\end{equation}
Then we have
\[
\begin{split}
\mathcal{B}(k_1x\circ y+k_2x\ast y, z)=&\langle\varphi (-(k_1\mathcal L_\circ(x)+k_2\mathcal L_\circ(x))y), z\rangle\\
=&-\langle(k_1{\rm
ad}_{1}^*(x)+k_2{\rm
ad}_{2}^*(x))\varphi(y), z\rangle\\
=&-k_1\langle{\rm
ad}_{1}^*(x)\varphi(y), z\rangle-k_2\langle{\rm
ad}_{2}^*(x)\varphi(y), z\rangle\\
=&k_1\langle\varphi(y),[x,z]_{1}\rangle+k_2\langle\varphi(y),[x,z]_{2}\rangle\\
=&\langle\varphi(y),k_1[x,z]_{1}+k_2[x,z]_{2}\rangle\\
=&\langle\varphi(y),[x,z]_{\g}\rangle\\
=&\mathcal{B}(y,[x,z]_{\g}),\;\;\forall x,y,z\in A.    
\end{split}
\]
Thus $\mathcal{B}$ is invariant on $(A,\circ, \ast)$.

Conversely, suppose that $\mathcal{B}$ is a nondegenerate
invariant bilinear form on $(A,\circ, \ast)$. So is $\mathcal{B}$  on $(A,\circ)$ and $(A,\ast)$. Then by Proposition 2.24 in \cite{LB}, $(-\mathcal{L}_{\circ}, A)$  and $({\rm
ad}_{1}^*, A)$ are equivalent. So are $(-\mathcal{L}_{\ast}, A)$  and $({\rm
ad}_{2}^*, A)$. Define a linear map
$\varphi:A\rightarrow A^{*}$ by Eq.~(\ref{eq:pro:equivalence and
invariance}). Then by a similar proof as above, we  show that
$\varphi$ gives an equivalence between $(-\mathcal{L}_{\circ},-\mathcal{L}_{\ast},A)$
and $(\mathrm{ad}_{1}^{*}, \mathrm{ad}_{2}^{*},A^{*})$ as representations of $(\g(A),[-,-]_{1}, [-,-]_{2})$.
\end{proof}

From Proposition~\ref{pro:equivalence and invariance}, Corollary~\ref{cor:corres} and their proofs, we can conclude that for a compatible Lie algebra $(\g,[-,-]_{1}, [-,-]_{2})$, if there is a nondegenerate commutative $2$-cocycle
on $(\g,[-,-]_{1}, [-,-]_{2})$, then there is a compatible anti-pre-Lie algebra $(\g,\circ, \ast)$ given by
Eq.~(\ref{eq:thm:commutative $2$-cocycles and anti-pre-Lie
algebras}). Moreover, $(-\mathcal{L}_{\circ}, -\mathcal{L}_{\ast}, \g)$ and $(\mathrm{ad}_{1}^{*},\mathrm{ad}_{2}^{*},\mathfrak{g}^{*})$ are
equivalent as representations of $(\g,[-,-]_{1}, [-,-]_{2})$. Conversely, if $(\g,\circ, \ast)$ is a  compatible anti-pre-Lie algebra  such that
$(-\mathcal{L}_{\circ}, -\mathcal{L}_{\ast}, \mathfrak{g})$ and $(\mathrm{ad}_{1}^{*}, \mathrm{ad}_{2}^{*},\mathfrak{g}^{*})$ are
equivalent as representations of $(\g,[-,-]_{1}, [-,-]_{2})$, then there is a nondegenerate bilinear form $\mathcal B$ satisfying
\begin{equation*}
\mathcal B([x,y]_{\g},z)+\mathcal B(y,[z,x]_{\g})+\mathcal B(x,[y,z]_{\g})=0,\;\;\forall x,y,z\in \g.
\end{equation*}

\

Let $(\rho, \mu, V)$ be a representation of a compatible Lie algebra $(\g,[-,-]_{1}, [-,-]_{2})$. Then there is a Lie algebra structure on  the direct
sum $\g\oplus V$ of vector spaces which is called the {\bf semi-direct
product}  (see \cite{WB}) defined by
\[
\begin{split}
&[x+u,y+v]_{1}=[x,y]_{1}+\rho(x)v-\rho(y)u,\\
&[x+u,y+v]_{2}=[x,y]_{2}+\mu(x)v-\mu(y)u,\;\;\forall x,y\in \g,
u,v\in V.   
\end{split}
\]
It is denoted by $\g\ltimes_{\rho, \mu}V$. Furthermore, there is the
following construction of
nondegenerate commutative $2$-cocycles from compatible anti-pre-Lie algebras.

\begin{prop}
Let $(A,\circ, \ast)$ be a compatible anti-pre-Lie algebra and $(A,[-,-]_{1}, [-,-]_{2})$ be the sub-adjacent compatible Lie algebra. Define a bilinear
form $\mathcal{B}$ on $A\oplus A^{*}$ by
\begin{equation}\label{eq:pro:commutative $2$-cocycles on semi-direct Lie algebras1}
\mathcal{B}(x+a^{*},y+b^{*})=\langle x,b^{*}\rangle+\langle
a^{*},y\rangle, \forall x,y\in A, a^{*},b^{*}\in A^{*}.
\end{equation}
Then $\mathcal{B}$ is a nondegenerate commutative $2$-cocycle on the
Lie algebra $A\ltimes_{-\mathcal{L}^{*}_{\circ}, -\mathcal{L}^{*}_{\ast}}A^*$.
Conversely, let $(\g,[-,-]_{1},[-,-]_{2})$ be a compatible Lie algebra and
$(\rho, \mu,\g^*)$ be a representation. Suppose that the bilinear
form given by Eq.~(\ref{eq:pro:commutative $2$-cocycles on
semi-direct Lie algebras1}) is a commutative $2$-cocycle on $\g\ltimes_{\rho, \mu}\g^{*}$. Then there is a structure of a  compatible
anti-pre-Lie algebra $(\g,\circ, \ast)$ on $(\g,[-,-]_{1}, [-,-]_{2})$ such that $\rho=-\mathcal{L}^{*}_{\circ}, \mu=-\mathcal{L}^{*}_{\ast}$.
\end{prop}

\begin{proof}
By Proposition 2.26 in \cite{LB} $\mathcal{B}$ is a nondegenerate commutative $2$-cocycle on the Lie algebras $A\ltimes_{-\mathcal{L}^{*}_{\circ}}A^*$ and $A\ltimes_{ -\mathcal{L}^{*}_{\ast}}A^*$. Let $x,y,z\in
A, a^{*},b^{*},c^{*}\in A^{*}$. Then
we have
\[
\begin{split}
\mathcal{B}([x+a^{*},y+b^{*}]_A,z+c^{*})=&\mathcal{B}([x,y]_{A}-k_1\mathcal{L}^{*}_{\circ}(x)b^{*}+k_1\mathcal{L}^{*}_{\circ}(y)a^{*}-k_2\mathcal{L}^{*}_{\ast}(x)b^{*}+k_2\mathcal{L}^{*}_{\ast}(y)a^{*},z+c^{*})\\
=&k_1\langle [x,y]_{1},c^{*}\rangle-\langle k_1\mathcal{L}^{*}_{\circ}(x)b^{*},z\rangle+\langle k_1\mathcal{L}^{*}_{\circ}(y)a^{*},z\rangle\\
&+k_2\langle [x,y]_{2},c^{*}\rangle-\langle k_2\mathcal{L}^{*}_{\ast}(x)b^{*},z\rangle+\langle k_2\mathcal{L}^{*}_{\ast}(y)a^{*},z\rangle\\
=&k_1(\langle[x,y]_A,c^{*}\rangle+\langle b^{*},x\circ z\rangle-\langle
a^{*},y\circ z\rangle)\\
&+k_2(\langle[x,y]_2,c^{*}\rangle+\langle b^{*},x\ast z\rangle-\langle
a^{*},y\ast z\rangle)
\end{split}
\]
Similarly,
\begin{equation*}
\begin{split}
\mathcal{B}([y+b^{*},z+c^{*}]_A,x+a^{*})=&k_1(\langle
[y,z]_1,a^{*}\rangle+\langle c^{*},y\circ x\rangle-\langle
b^{*},z\circ x\rangle)\\
&+k_2(\langle
[y,z]_2,a^{*}\rangle+\langle c^{*},y\ast x\rangle-\langle
b^{*},z\ast x\rangle),\\
\mathcal{B}([z+c^{*},x+a^{*}]_A,y+b^{*})=&k_1(\langle[z,x]_1,b^{*}\rangle+\langle
a^{*},z\circ y\rangle-\langle c^{*},x\circ y\rangle)\\
&+k_2(\langle[z,x]_2,b^{*}\rangle+\langle
a^{*},z\ast y\rangle-\langle c^{*},x\ast y\rangle).
\end{split}
\end{equation*}
 Thus
$$\mathcal{B}([x+a^{*},y+b^{*}]_A,z+c^{*})+\mathcal{B}([y+b^{*},z+c^{*}]_A,x+a^{*})+\mathcal{B}([z+c^{*},x+a^{*}]_A,y+b^{*})=0.$$
Hence $\mathcal{B}$ is a commutative $2$-cocycle on the Lie algebra $A\ltimes_{-\mathcal{L}^{*}_{\circ}, -\mathcal{L}_{\ast}}A^{*}$.

Conversely, by Theorem~\ref{thm:commutative $2$-cocycles and anti-pre-Lie
algebras}, there is a structure of a  compatible anti-pre-Lie algebra 
given by  bilinear operations $\circ$ and $\ast$ on $\g\ltimes_{\rho}\g^{*}$ and $\g\ltimes_{\mu}\g^{*}$, repectively, defined by
Eq.~(\ref{eq:thm:commutative $2$-cocycles and anti-pre-Lie
algebras}). In particular, we have
\[
\begin{split}
\mathcal B(x\circ y,z)&=\mathcal B(y,[x,z]_{1})=0,\\
\mathcal B(x\ast y,z)
&=\mathcal B(y,[x,z]_{2})=0,\;\;\forall
x,y,z\in \g.
\end{split}
\]
So for all $x,y\in \g$, $x\circ y, x\ast y \in
\g$ and thus $(\g,\circ, \ast)$ is a compatible anti-pre-Lie
algebra. Moreover,
\begin{equation*}
\langle y, -\mathcal L^*_\circ(x) a^*\rangle=\langle x\circ y,
a^{*}\rangle=\mathcal B(x\circ y, a^*)=\mathcal
B(y,[x,a^*]_{1})=\langle y,\rho(x)a^{*}\rangle, \;\; \forall x,y\in
\g, a^{*}\in \g^{*}.
\end{equation*}
Hence $\rho=-\mathcal{L}^{*}_{\circ}$.

Similarly,
\begin{equation*}
\langle y, -\mathcal L^*_\ast(x) a^*\rangle=\langle x\ast y,
a^{*}\rangle=\mathcal B(x\ast y, a^*)=\mathcal
B(y,[x,a^*]_{2})=\langle y,\mu(x)a^{*}\rangle, \;\; \forall x,y\in
\g, a^{*}\in \g^{*}.
\end{equation*}
Hence $\mu=-\mathcal{L}^{*}_{\ast}$.
\end{proof}

At the end of this section, we consider a construction of compatible
anti-pre-Lie algebras from symmetric bilinear forms, where in
particular these bilinear forms are invariant.

\begin{prop}
Let $\mathcal{B}$ be a symmetric bilinear form on a vector space
$A$ and $s_1, s_2\in A$ be  fixed vectors. Define  bilinear operations
$\circ, \ast:A\otimes A\rightarrow A$ by
\begin{equation*}
\begin{split}
x\circ y&=\mathcal{B}(x,y)s_1-\mathcal{B}(x,s_1)y,\\
x\ast y&=\mathcal{B}(x,y)s_2-\mathcal{B}(x,s_2)y,\;\;\forall x,y\in A.
\end{split}
\end{equation*}
 Then $(A,\circ, \ast)$ is a compatible anti-pre-Lie
algebra. Moreover, $\mathcal{B}$ is invariant on $(A,\circ, \ast)$ and
thus $\mathcal{B}$ is a commutative $2$-cocycle on the sub-adjacent compatible
Lie algebra $(A,[-,-]_{1}, [-,-]_{2})$.
\end{prop}

\begin{proof}
By Proposition 2.27 in \cite{LB}, $(A, \circ)$ and $(A, \ast)$ are anti-pre-Lie algebras. Now it suffices to show \eqref{compatible anti-pre id1} and \eqref{compatible anti-pre id2}.
Let $x,y,z\in A$. Then 
we have
\begin{eqnarray*}
x\circ(y\ast z)&=&\mathcal{B}(y,z)x\circ s_2-\mathcal{B}(y,s_2)x\circ z\\
&=&\mathcal{B}(y,z)\mathcal{B}(x,s_2)s_1-\mathcal{B}(y,z)\mathcal{B}(x,s_1)s_2-\mathcal{B}(y,s_2)\mathcal{B}(x,z)s_1+\mathcal{B}(y,s_2)\mathcal{B}(x,s_1)z,\\
\mbox{}[x,y]_{2}\circ z&=&(\mathcal{B}(x,y)s_2-\mathcal{B}(x,s_2)y-\mathcal{B}(y,x)s_2+\mathcal{B}(y,s_2)x)\circ
z\\
&=&-\mathcal{B}(x,s_2)y\circ z+\mathcal{B}(y,s_2)x\circ z\\
&=&
-\mathcal{B}(x,s_2)\mathcal{B}(y,z)s_1+\mathcal{B}(x,s_2)\mathcal{B}(y,s_1)z+\mathcal{B}(y,s_2)\mathcal{B}(x,z)s_1-\mathcal{B}(y,s_2)\mathcal{B}(x,s_1)z.
\end{eqnarray*}
By changing the position of $\circ$ and $\ast$, and by swapping $x$ and $y$ we have
\[
\begin{split}
\text{LHS of \eqref{compatible anti-pre id1}}=&\mathcal{B}(y,z)\mathcal{B}(x,s_2)s_1-\mathcal{B}(y,z)\mathcal{B}(x,s_1)s_2-\mathcal{B}(y,s_2)\mathcal{B}(x,z)s_1+\mathcal{B}(y,s_2)\mathcal{B}(x,s_1)z\\
&+\mathcal{B}(y,z)\mathcal{B}(x,s_1)s_2-\mathcal{B}(y,z)\mathcal{B}(x,s_2)s_1-\mathcal{B}(y,s_1)\mathcal{B}(x,z)s_2+\mathcal{B}(y,s_1)\mathcal{B}(x,s_2)z\\
&-\mathcal{B}(x,z)\mathcal{B}(y,s_2)s_1+\mathcal{B}(x,z)\mathcal{B}(y,s_1)s_2+\mathcal{B}(x,s_2)\mathcal{B}(y,z)s_1-\mathcal{B}(x,s_2)\mathcal{B}(y,s_1)z\\
&-\mathcal{B}(x,z)\mathcal{B}(y,s_1)s_2+\mathcal{B}(x,z)\mathcal{B}(y,s_2)s_1+\mathcal{B}(x,s_1)\mathcal{B}(y,z)s_2-\mathcal{B}(x,s_1)\mathcal{B}(y,s_2)z\\
=&-\mathcal{B}(y,s_2)\mathcal{B}(x,z)s_1+\mathcal{B}(y,s_2)\mathcal{B}(x,s_1)z-\mathcal{B}(y,s_1)\mathcal{B}(x,z)s_2+\mathcal{B}(y,s_1)\mathcal{B}(x,s_2)z\\
&+\mathcal{B}(x,s_2)\mathcal{B}(y,z)s_1-\mathcal{B}(x,s_2)\mathcal{B}(y,s_1)z+\mathcal{B}(x,s_1)\mathcal{B}(y,z)s_2-\mathcal{B}(x,s_1)\mathcal{B}(y,s_2)z\\
=&\text{RHS of \eqref{compatible anti-pre id1}},
\end{split}
\]
\[
\begin{split}
\text{LHS of \eqref{compatible anti-pre id2}}=&-\mathcal{B}(x,s_2)\mathcal{B}(y,z)s_1+\mathcal{B}(x,s_2)\mathcal{B}(y,s_1)z+\mathcal{B}(y,s_2)\mathcal{B}(x,z)s_1-\mathcal{B}(y,s_2)\mathcal{B}(x,s_1)z\\
&-\mathcal{B}(x,s_1)\mathcal{B}(y,z)s_2+\mathcal{B}(x,s_1)\mathcal{B}(y,s_2)z+\mathcal{B}(y,s_1)\mathcal{B}(x,z)s_2-\mathcal{B}(y,s_1)\mathcal{B}(x,s_2)z\\
&-\mathcal{B}(y,s_2)\mathcal{B}(z,x)s_1+\mathcal{B}(y,s_2)\mathcal{B}(z,s_1)x+\mathcal{B}(z,s_2)\mathcal{B}(y,x)s_1-\mathcal{B}(z,s_2)\mathcal{B}(y,s_1)x\\
&-\mathcal{B}(y,s_1)\mathcal{B}(z,x)s_2+\mathcal{B}(y,s_1)\mathcal{B}(z,s_2)x+\mathcal{B}(z,s_1)\mathcal{B}(y,x)s_2-\mathcal{B}(z,s_1)\mathcal{B}(y,s_2)x\\
&-\mathcal{B}(z,s_2)\mathcal{B}(x,y)s_1+\mathcal{B}(z,s_2)\mathcal{B}(x,s_1)y+\mathcal{B}(x,s_2)\mathcal{B}(z,y)s_1-\mathcal{B}(x,s_2)\mathcal{B}(z,s_1)y\\
&-\mathcal{B}(z,s_1)\mathcal{B}(x,y)s_2+\mathcal{B}(z,s_1)\mathcal{B}(x,s_2)y+\mathcal{B}(x,s_1)\mathcal{B}(z,y)s_2-\mathcal{B}(x,s_1)\mathcal{B}(z,s_2)y\\
=&0.
\end{split}
\]
Hence $(A,\circ, \ast)$ is a compatible anti-pre-Lie algebra. Moreover, $\mathcal B$ is invariant on $(A,\circ)$ and $(A,\ast)$. Thus we have
\[
\begin{split}
\mathcal{B}(y,[x,z]_{\g})=&\mathcal{B}(y,k_1[x,z]_{1})+\mathcal{B}(y,k_2[x,z]_{2})\\
=&k_1\mathcal{B}(x\circ y,z)+k_2\mathcal{B}(x\ast y,z), \;\;\forall x,y,z\in A.
\end{split}
\]
Therefore $\mathcal B$ is invariant on $(A,\circ, \ast)$.
\end{proof}



    

\section{Classification of 2-dimensional compatible anti-pre-Lie algebras}

The classification of
2-dimensional complex non-commutative anti-pre-Lie algebras is given in \cite{LB} and by Proposition 2.6 in that paper, the commutative anti-pre-Lie algebras
are commutative associative algebras whose classification is
known (\cite{Bai2001} or~\cite{Burde1998}). 

Now we present the classification of anti-pre-Lie algebra of dimension two.

\begin{prop}\label{thm:classification}
Let $(A,\circ)$ be a $2$-dimensional anti-pre-Lie algebra over the complex numbers field
$\mathbb{C}$ with a basis $\{e_{1},e_{2}\}$. Then $(A,\circ)$ is isomorphic to one of the following  non-isomorphic algebras:
\end{prop}
\begin{enumerate}

\item[$A_1$:] Abelian;
\item[$A_2$:] $e_1\circ e_1=e_1$;
\item[$A_3$:] $e_1\circ e_1=e_2$;
\item[$A_4$:] $e_1\circ e_1=e_1, \ e_2\circ e_2=e_2$;
\item[$A_5$:] $e_{1}\circ e_{1}=-e_{2}, e_{1}\circ
e_{2}=0, e_{2}\circ e_{1}=-e_{1}, e_{2}\circ e_{2}=0;$

\item[$A_6(\lambda)$:] ($\lambda\in\mathbb{C}$)   $e_{1}\circ
e_{1}=0, e_{1}\circ e_{2}=0, e_{2}\circ e_{1}=-e_{1}, e_{2}\circ
e_{2}=\lambda e_{2};$

\item[$A_7$:] $e_{1}\circ e_{1}=0, e_{1}\circ e_{2}=0,
e_{2}\circ e_{1}=-e_{1}, e_{2}\circ e_{2}=e_{1}-e_{2};$

\item[$A_8(\lambda)$:] ($\lambda\ne -1$)  $e_{1}\circ e_{1}=0,
e_{1}\circ e_{2}=(\lambda+1)e_{1}, e_{2}\circ e_{1}=\lambda e_{1},
e_{2}\circ e_{2}=(\lambda-1)e_{2}$;

\item[$A_9$:] $e_{1}\circ e_{1}=0, e_{1}\circ
e_{2}=-e_{1}, e_{2}\circ e_{1}=-2e_{1}, e_{2}\circ
e_{2}=e_{1}-3e_{2}.$
\end{enumerate}

Algebraic classification of nilpotent compatible Lie algebras up to dimension four is given in \cite{LLL}. The authors of \cite{AKM} gave the algebraic and geometric classsification of compatible pre-Lie algebras. In both papers, the used methods are similar. In this section,  we use that method with a little modification to classify 2-dimensional compatible anti-pre-Lie algebras.

Now, let us recall the method to classify compatible algebras.

Let $(A,\circ)$ be an algebra. Then the method consists of four steps:

\textbf{Step 1.} Compute a set $\mathrm{Z}^2(A,A)$ of all bilinear maps $\phi : A\times A \rightarrow A$ satisfying some necessary equalities. In our case they are
\begin{align*}
\text{i.}\quad&\phi(x,\phi(y,z))-\phi(y,\phi(x,z))=\phi(\phi(y,x),z)-\phi(\phi(x,y),z),\\
\text{ii.}\quad&\phi(\phi(x,y),z)-\phi(\phi(y,x),z)+cyclic=0,\\
\text{iii.}\quad&x\circ\phi(y,z)+\phi(x,y\circ z)-y\circ\phi(x,z)-\phi(y,x\circ z)=\phi(y,x)\circ z-\phi(x,y)\circ z+\phi(y\circ x,z)-\phi(x\circ y, z),\\
\text{iv.}\quad&\phi(x,y)\circ z-\phi(y,x)\circ z+\phi(x\circ y,z)-\phi(y\circ x,z)+cyclic=0
\end{align*}

\textbf{Step 2.} Since $\phi=0\in \mathrm{Z}^2(A,A)$,  $\mathrm{Z}^2(A,A) \neq \emptyset$ . For $\phi\in \mathrm{Z}^2(A,A)$,  define a multiplication $\ast_{\phi}$ on $A$ by $x\ast_{\phi} y = \phi(x, y)$ for all $x, y \in A$.
Then $(A,\circ,\ast_{\phi})$ is a compatible  algebra.

\textbf{Step 3.} Find the orbits of $Aut(A)$ on $\mathrm{Z}^2(A,A)$.

\textbf{Step 4.} Choose a representative $\phi$ from each orbit to classify the constructed compatible 
algebra $(A,\circ,\ast)$.

\begin{lem}\label{cocycles}
Let $(A,\circ)$ be an anti-pre-Lie algebra in Proposition \ref{thm:classification}. The set $\mathrm{Z}^2(A,A)$ is of the form:
\end{lem}
(1) \ $\mathrm{Z}^2(A_2,A_2):$
\[
\begin{cases}
\phi(e_1,e_1)=\alpha e_1+\beta e_2,\\
\phi(e_2,e_1)=\gamma e_1,\\
\phi(e_1,e_2)=\gamma e_1,\\
\phi(e_2,e_2)=\gamma e_2.
\end{cases} \quad
\begin{cases}
\phi(e_1,e_1)=\alpha e_1+\beta e_2,\\
\phi(e_2,e_1)=0,\\
\phi(e_1,e_2)=\gamma e_2,\\
\phi(e_2,e_2)=0.
\end{cases} \quad
\begin{cases}
\phi(e_1,e_1)=\alpha e_1,\\
\phi(e_2,e_1)=0,\\
\phi(e_1,e_2)=0,\\
\phi(e_2,e_2)=\beta e_2.
\end{cases}
\]

(2) \ $\mathrm{Z}^2(A_3,A_3):$
\[
\begin{cases}
\phi(e_1,e_1)=\alpha e_1+\beta e_2,\\
\phi(e_2,e_1)=0,\\
\phi(e_1,e_2)=\gamma e_2,\\
\phi(e_2,e_2)=0.
\end{cases}
\quad
\begin{cases}
\phi(e_1,e_1)=\alpha e_1+\beta e_2,\\
\phi(e_2,e_1)=\gamma e_1+\delta e_2,\\
\phi(e_1,e_2)=0,\\
\phi(e_2,e_2)=0.
\end{cases}
\quad
\begin{cases}
\phi(e_1,e_1)=\alpha e_1+\beta e_2,\\
\phi(e_2,e_1)=\gamma e_2,\\
\phi(e_1,e_2)=\delta e_2,\\
\phi(e_2,e_2)=0.
\end{cases}
\quad
\begin{cases}
\phi(e_1,e_1)=\alpha e_1+\beta e_2,\\
\phi(e_2,e_1)=\gamma e_1,\\
\phi(e_1,e_2)=\gamma e_1,\\
\phi(e_2,e_2)=\gamma e_2.
\end{cases}
\]

(3) \ $\mathrm{Z}^2(A_4,A_4):$
\[
\begin{cases}
\phi(e_1,e_1)=\alpha e_1-\beta e_2,\\
\phi(e_2,e_1)=-\gamma e_1+\beta e_2,\\
\phi(e_1,e_2)=-\gamma e_1+\beta e_2,\\
\phi(e_2,e_2)=\gamma e_1+\delta e_2.
\end{cases}
\]

(4) \ $\mathrm{Z}^2(A_5,A_5):$
\[
\begin{cases}
\phi(e_1,e_1)=-\alpha e_1+\beta e_2,\\
\phi(e_2,e_1)=\gamma e_1+\alpha e_2,\\
\phi(e_1,e_2)=0,\\
\phi(e_2,e_2)=0.
\end{cases}
\quad 
\begin{cases}
\phi(e_1,e_1)=\alpha e_2,\\
\phi(e_2,e_1)=\alpha e_1,\\
\phi(e_1,e_2)=\beta e_2,\\
\phi(e_2,e_2)=\beta e_1.
\end{cases}
\]

(5) \ $\mathrm{Z}^2(A_6,A_6):$

If $\lambda=0$, then
\[
\begin{cases}
\phi(e_1,e_1)=-\alpha
 e_1+\beta e_2,\\
\phi(e_2,e_1)=\gamma e_1+\alpha e_2,\\
\phi(e_1,e_2)=0,\\
\phi(e_2,e_2)=0.
\end{cases}
\quad 
\begin{cases}
\phi(e_1,e_1)=0,\\
\phi(e_2,e_1)=\alpha e_2,\\
\phi(e_1,e_2)=0,\\
\phi(e_2,e_2)=\beta e_2.
\end{cases}
\quad 
\begin{cases}
\phi(e_1,e_1)=0,\\
\phi(e_2,e_1)=\alpha e_1,\\
\phi(e_1,e_2)=0,\\
\phi(e_2,e_2)=\beta e_1+\gamma e_2.
\end{cases}
\]

If $\lambda\neq-1$, then
\[
\begin{cases}
\phi(e_1,e_1)=0,\\
\phi(e_2,e_1)=\alpha e_1,\\
\phi(e_1,e_2)=\beta e_1 +\gamma e_2,\\
\phi(e_2,e_2)=0.
\end{cases}
\]

If $\lambda=-1$, then
\[
\begin{cases}
\phi(e_1,e_1)=\alpha e_1,\\
\phi(e_2,e_1)=\beta e_1,\\
\phi(e_1,e_2)=\alpha e_2, \\
\phi(e_2,e_2)=\beta e_2.
\end{cases}
\]

(6) \ $\mathrm{Z}^2(A_7,A_7):$
\[
\begin{cases}
\phi(e_1,e_1)=0,\\
\phi(e_2,e_1)=\alpha e_1,\\
\phi(e_1,e_2)=0, \\
\phi(e_2,e_2)=\beta e_1+\gamma e_2.
\end{cases}
\]

(7) \ $\mathrm{Z}^2(A_8,A_8):$
If $\lambda\neq\{-2,0\}$, then
\[
\begin{cases}
\phi(e_1,e_1)=0,\\
\phi(e_2,e_1)=(\alpha+\beta) e_1,\\
\phi(e_1,e_2)=2\alpha e_1, \\
\phi(e_2,e_2)=\gamma e_1+2\beta e_2.
\end{cases}
\]

If $\lambda=-2$, then
\[
\begin{cases}
\phi(e_1,e_1)=3\alpha e_1,\\
\phi(e_2,e_1)=2\beta e_1+\alpha e_2,\\
\phi(e_1,e_2)=\gamma e_1+2\alpha e_2,\\
\phi(e_2,e_2)=3\beta e_2.
\end{cases}
\]

If $\lambda=0$, then
\[
\begin{cases}
\phi(e_1,e_1)=0,\\
\phi(e_2,e_1)=(\alpha+\beta)e_1,\\
\phi(e_1,e_2)=2\alpha e_1, \\
\phi(e_2,e_2)=\gamma e_1+2\beta e_2.
\end{cases}
\quad
\begin{cases}
\phi(e_1,e_1)=0,\\
\phi(e_2,e_1)=0,\\
\phi(e_1,e_2)=\alpha e_1+\beta e_2, \\
\phi(e_2,e_2)=\gamma e_1+\delta e_2.
\end{cases}
\quad
\begin{cases}
\phi(e_1,e_1)=-\alpha e_1,\\
\phi(e_2,e_1)=\alpha e_2,\\
\phi(e_1,e_2)=-\beta e_1, \\
\phi(e_2,e_2)=\beta e_2.
\end{cases}
\]

(8) \ $\mathrm{Z}^2(A_9,A_9):$
\[
\begin{cases}
\phi(e_1,e_1)=0,\\
\phi(e_2,e_1)=(\alpha +\beta)e_1,\\
\phi(e_1,e_2)=2\alpha e_1, \\
\phi(e_2,e_2)=\gamma e_1+2\beta e_2.
\end{cases}
\]

\begin{lem}
The description of the group of automorphisms of every $2$-dimensional anti-pre-Lie
algebra is given below.

\centering{\begin{tabular}{|l|p{3in}|}
\hline    
\text{Algebra} & \text{Automorphisms}  \\
\hline\hline
$A_1$& $\theta(e_1)=a e_1+b e_2, \quad \theta(e_2)=c e_1 + d e_2$, \newline
where $ad-bc\neq0$.\\
\hline
$A_2$&$\theta(e_1)=e_1, \quad \theta(e_2)=a e_2$, \ $a\in\mathbb{C}^{*}$ .\\
\hline
$A_3$& $\theta(e_1)=a e_1+b e_2, \quad \theta(e_2)=a^2 e_2$, \ $a\in\mathbb{C}^{*}, \ b\in\mathbb{C}$.\\
\hline
$A_4$&$\theta(e_1)=e_1, \quad \theta(e_2)= e_2$ or \newline  $\theta(e_1)=e_2, \quad \theta(e_2)= e_1$.\\
\hline
$A_5$&$\theta(e_1)=e_1, \quad \theta(e_2)= e_2$ or\newline $\theta(e_1)=-e_1, \quad \theta(e_2)= e_2$.\\
\hline
$A_6(\lambda)$&$\theta(e_1)=a e_1, \quad \theta(e_2)= e_2$, \ $a\in\mathbb{C}^{*}$.\\
\hline
$A_7$&$\theta(e_1)=e_1, \quad  \theta(e_2)= a e_1+e_2$, \ $a\in\mathbb{C}$.\\
\hline
$A_8(\lambda), \ \lambda\neq-1$&$\theta(e_1)=a e_1, \quad \theta(e_2)= e_2$, \ $a\in\mathbb{C}^{*}$.\\
\hline
$A_9$&$\theta(e_1)= e_1, \quad \theta(e_2)= a e_1+e_2$, $a\in\mathbb{C}$.\\
\hline
\end{tabular}}
\end{lem}

\begin{thm}
Let $A$ be a nonzero $2$-dimensional compatible anti-pre-Lie algebra. Then $A$ is isomorphic to one and only one of the following algebras:
\end{thm}
\begin{enumerate}
\item[$CA_1$:] Abelian;
\item[$CA_2$:] $e_1\circ e_1=e_1$;
\item[$CA_3$:] $e_1\circ e_1=e_2$;
\item[$CA_4$:] $e_1\circ e_1=e_1, \ e_2\circ e_2=e_2$;
\item[$CA_5$:] $e_{1}\circ e_{1}=-e_{2}, e_{1}\circ
e_{2}=0, e_{2}\circ e_{1}=-e_{1}, e_{2}\circ e_{2}=0;$

\item[$CA_6(\lambda)$:] ($\lambda\in\mathbb{C}$)   $e_{1}\circ
e_{1}=0, e_{1}\circ e_{2}=0, e_{2}\circ e_{1}=-e_{1}, e_{2}\circ
e_{2}=\lambda e_{2};$

\item[$CA_7$:] $e_{1}\circ e_{1}=0, e_{1}\circ e_{2}=0,
e_{2}\circ e_{1}=-e_{1}, e_{2}\circ e_{2}=e_{1}-e_{2};$

\item[$CA_8(\lambda)$:] ($\lambda\ne -1$)  $e_{1}\circ e_{1}=0,
e_{1}\circ e_{2}=(\lambda+1)e_{1}, e_{2}\circ e_{1}=\lambda e_{1},
e_{2}\circ e_{2}=(\lambda-1)e_{2}$;

\item[$CA_9$:] $e_{1}\circ e_{1}=0, e_{1}\circ
e_{2}=-e_{1}, e_{2}\circ e_{1}=-2e_{1}, e_{2}\circ
e_{2}=e_{1}-3e_{2}.$

\item[$CA_{10}(\alpha, \beta)$:] 
$\begin{cases}
e_1\circ e_1=e_1,\\
e_1\ast e_1=\alpha e_1+\beta e_2, \ e_2\ast e_1= e_1, \ e_1\ast e_2= e_1, \ e_2\ast e_2= e_2.
\end{cases}$

\item[$CA_{11}(\alpha)$:] $\begin{cases}
e_1\circ e_1=e_1,\\
e_1\ast e_1=\alpha e_1+\delta e_2, \quad \delta\in\{0,1\}.
\end{cases}$

\item[$CA_{12}(\alpha, \gamma)$:] $\begin{cases}
e_1\circ e_1=e_1,\\
e_1\ast e_1=\alpha e_1+\delta e_2, \ e_1\ast e_2=\gamma e_2, \quad \text{where} \quad \delta\in\{0,1\}.
\end{cases}$

\item[$CA_{13}(\alpha)$:] $\begin{cases}
e_1\circ e_1=e_1,\\
e_1\ast e_1=\alpha e_1, \ e_2\ast e_2=\delta e_2, \quad \delta\in\{0,1\}.
\end{cases}$

\item[$CA_{14}(\beta)$:] $\begin{cases}
e_1\circ e_1=e_2,\\
e_1\ast e_1=\delta e_1+\beta e_2, \ e_1\ast e_2=\delta e_2, \quad \delta\in\{0,1\}.
\end{cases}$

\item[$CA_{15}$:] $\begin{cases}
e_1\circ e_1=e_2,\\
e_1\ast e_2=e_2.
\end{cases}$

\item[$CA_{16}(\gamma)$:] $\begin{cases}
e_1\circ e_1=e_2,\\
e_1\ast e_1=e_1, \ e_1\ast e_2=\gamma e_2, \quad \gamma\neq1.
\end{cases}$

\item[$CA_{17}(\beta, \delta)$:] $\begin{cases}
e_1\circ e_1=e_2,\\
e_1\ast e_1=\beta e_2, \ e_2\ast e_1=e_1+\delta e_2.
\end{cases}$

\item[$CA_{18}( \delta)$:] $\begin{cases}
e_1\circ e_1=e_2,\\
e_1\ast e_1=e_1, \ e_2\ast e_1=\delta e_2,
\end{cases}$

\item[$CA_{19}$:] $\begin{cases}
e_1\circ e_1=e_2,\\
e_2\ast e_1=e_2,
\end{cases}$

\item[$CA_{20}( \beta)$:] $\begin{cases}
e_1\circ e_1=e_2,\\
e_1\ast e_1=\delta e_1+\beta e_2, \ e_2\ast e_1=\delta e_2, \quad \delta\in\{0,1\}.
\end{cases}$

\item[$CA_{21}(\alpha, \gamma, \delta)$:] $\begin{cases}
e_1\circ e_1=e_2,\\
e_1\ast e_1=\alpha e_1, \ e_2\ast e_1=\gamma e_2, \ e_1\ast e_2=\delta e_2.
\end{cases}$

\item[$CA_{22}(\beta)$:] $\begin{cases}
e_1\circ e_1=e_2,\\
e_1\ast e_1=\beta e_2, \ e_2\ast e_1=-\delta e_2, \ e_1\ast e_2=\delta e_2, \quad \delta\in\{0,1\}.
\end{cases}$

\item[$CA_{23}(\beta, \gamma, \delta)$:] $\begin{cases}
e_1\circ e_1=e_2,\\
e_1\ast e_1=e_1+\beta e_2, \ e_2\ast e_1=\gamma e_2, \ e_1\ast e_2=\delta e_2.
\end{cases}$

\item[$CA_{24}(\beta)$:] $\begin{cases}
e_1\circ e_1=e_2,\\
e_1\ast e_1=\beta e_2.
\end{cases}$

\item[$CA_{25}$:] $\begin{cases}
e_1\circ e_1=e_2,\\
e_1\ast e_1=e_1.
\end{cases}$

\item[$CA_{26}(\beta)$:] $\begin{cases}
e_1\circ e_1=e_2,\\
e_1\ast e_1=\beta e_2, \ e_2\ast e_1= e_1, \ e_1\ast e_2=e_1, \ e_2\ast e_2= e_2.
\end{cases}$

\item[$CA_{27}(\alpha, \beta, \gamma,\delta)$:] $\begin{cases}
e_1\circ e_1=e_1, \ e_2\circ e_2=e_2,\\
e_1\ast e_1=\alpha e_1-\beta e_2, \ e_2\ast e_1=-\gamma e_1+\beta e_2, \ e_1\ast e_2=-\gamma e_1+\beta e_2, \ e_2\ast e_2=\gamma e_1+\delta e_2.
\end{cases}$

\item[$CA_{28}(\alpha, \beta, \gamma)$:] $\begin{cases}
e_1\circ e_1=-e_2, \ e_2\circ e_1=-e_1,\\
e_1\ast e_1=-\alpha e_1+\beta e_2, \ e_2\ast e_1=\gamma e_1+\alpha e_2.
\end{cases}$

\item[$CA_{29}(\alpha)$:] $\begin{cases}
e_1\circ e_1=-e_2, \ e_2\circ e_1=-e_1,\\
e_1\ast e_1=\alpha e_2, \ e_2\ast e_1=\alpha e_1, \ e_1\ast e_2=\beta e_2, \ e_2\ast e_2=\beta e_1.
\end{cases}$

\item[$CA_{30}(\beta, \gamma)$:] $\begin{cases}
e_2\circ e_1=-e_1, \\
e_1\ast e_1=-e_1+\beta e_2, \ e_2\ast e_1=\gamma e_1+e_2.
\end{cases}$

\item[$CA_{31}(\gamma)$:] $\begin{cases}
e_2\circ e_1=-e_1, \\
e_1\ast e_1=\delta e_2, \ e_2\ast e_1=\gamma e_1, \quad \delta\in\{0,1\}.
\end{cases}$

\item[$CA_{32}(\alpha,\beta)$:] $\begin{cases}
e_2\circ e_1=-e_1, \\
e_2\ast e_1=\alpha e_1, \ e_2\ast e_2=\beta  e_2.
\end{cases}$

\item[$CA_{33}(\alpha,\gamma)$:] $\begin{cases}
e_2\circ e_1=-e_1, \\
e_2\ast e_1=\alpha e_1, \ e_2\ast e_2=\delta e_1+\gamma e_2, \quad \delta\in\{0,1\}.
\end{cases}$

\item[$CA_{34}(\beta)$:] $\begin{cases}
e_2\circ e_1=-e_1, \ e_2\circ e_2=-e_2 \\
e_1\ast e_1=\delta e_1, \ e_2\ast e_1=\beta e_1, \ e_1\ast e_2=\delta e_2, \ e_2\ast e_2=\beta e_2, \quad \delta\in\{0,1\}.
\end{cases}$

\item[$CA_{35}(\alpha, \beta)$:] $\begin{cases}
e_2\circ e_1=-e_1, \ e_2\circ e_2=\lambda e_2 \\
e_2\ast e_1=\alpha e_1, \ e_1\ast e_2=\beta e_1+\delta e_2, \quad \delta\in\{0,1\}.
\end{cases}$

\item[$CA_{36}(\alpha, \gamma)$:] $\begin{cases}
e_2\circ e_1=-e_1, \ e_2\circ e_2=e_1-e_2 \\
e_2\ast e_1=\alpha e_1, \ e_2\ast e_2=\gamma e_2.
\end{cases}$

\item[$CA_{37}(\alpha, \beta)$:] $\begin{cases}
e_2\circ e_1=-e_1, \ e_2\circ e_2=e_1-e_2 \\
e_2\ast e_1=\alpha e_1, \ e_2\ast e_2=\beta e_1+\alpha e_2.
\end{cases}$

\item[$CA_{38}{[\lambda]}(\alpha, \beta)$:] $\begin{cases}
e_1\circ e_2=(\lambda+1)e_1, \ e_2\circ e_1=\lambda e_1, \ e_2\circ e_2=(\lambda-1)e_2 \\
e_2\ast e_1=(\alpha+\beta) e_1, \ e_1\ast e_2=2\alpha e_1, \ e_2\ast e_2=\delta e_1+2\beta e_2, \quad \delta\in\{0,1\}.
\end{cases}$

\item[$CA_{39}(\beta,\gamma)$:] $\begin{cases}
e_1\circ e_2=-e_1, \ e_2\circ e_1=-2e_1, \ e_2\circ e_2=-3e_2 \\
e_1\ast e_1=3\delta e_1,\ e_2\ast e_1=2\beta e_1+\delta e_2, \ e_1\ast e_2=\gamma e_1+2\delta e_2, \ e_2\ast e_2=3\beta e_2, \quad \delta\in\{0,1\}.
\end{cases}$

\item[$CA_{40}(\alpha,\beta)$:] $\begin{cases}
e_1\circ e_2=e_1, \  e_2\circ e_2=-e_2 \\ e_2\ast e_1=(\alpha +\beta) e_1, \ e_1\ast e_2=2\alpha e_1, \ e_2\ast e_2=\delta e_1+2\beta e_2, \quad \delta\in\{0,1\}.
\end{cases}$

\item[$CA_{41}(\alpha, \beta, \gamma)$:] $\begin{cases}
e_1\circ e_2=e_1, \  e_2\circ e_2=-e_2 \\ e_1\ast e_2=\alpha e_1+\beta e_2, \ e_2\ast e_2=\delta e_1+\gamma e_2, \quad \delta\in\{0,1\}.
\end{cases}$

\item[$CA_{42}(\alpha, \beta)$:] $\begin{cases}
e_1\circ e_2=e_1, \  e_2\circ e_2=-e_2 \\ e_1\ast e_2=\alpha e_1+\delta e_2, \ e_2\ast e_2=\beta e_2, \quad \delta\in\{0,1\}.
\end{cases}$

\item[$CA_{43}(\beta)$:] $\begin{cases}
e_1\circ e_2=e_1, \  e_2\circ e_2=-e_2 \\ e_1\ast e_1=-\delta e_1, \ e_2\ast e_1=\delta e_2, \ e_1\ast e_2=-\beta e_1, \ e_2\ast e_2=\beta e_2, \quad \delta\in\{0,1\}.
\end{cases}$

\item[$CA_{44}(\alpha, \beta)$:] $\begin{cases}
e_1\circ e_2=-e_1, \ e_2\circ e_1=-2 e_1, \ e_2\circ e_2=e_1-3e_2 \\
e_2\ast e_1=(\alpha+\beta) e_1, \ e_1\ast e_2=2\alpha e_1, \ e_2\ast e_2=2\beta e_2.
\end{cases}$

\item[$CA_{45}(\alpha, \gamma)$:] $\begin{cases}
e_1\circ e_2=-e_1, \ e_2\circ e_1=-2 e_1, \ e_2\circ e_2=e_1-3e_2 \\
e_2\ast e_1=2\alpha e_1, \ e_1\ast e_2=\alpha e_1, \ e_2\ast e_2=\gamma e_1-3\alpha e_2.
\end{cases}$
\end{enumerate}
Note that $CA_{27}(\alpha, \beta, \gamma,\delta)\cong CA_{27}(\delta, -\gamma, -\beta,\alpha)$, $CA_{28}(\alpha, \beta, \gamma)\cong CA_{28}(-\alpha, \beta, \gamma)$, $CA_{29}(\alpha, \beta)\cong CA_{29}(\alpha, -\beta)$. 
\begin{proof}

\

1. \textbf{Construction from $A_1$.}

Note that the constructed compatible anti-pre-Lie algebra $(A,\circ,\ast)$ by abelian anti-pre-Lie algebra $(A,\circ)$ can be considered as an anti-pre-Lie algebra $(A,\ast)$. Hence, we get algebras 
 from $CA_1$ to $CA_9$.

2. \textbf{Construction from $A_2$.}

By Lemma \ref{cocycles} we have the following cases:

\textbf{Case 1.}
\[
\begin{cases}
e_1\circ e_1=e_1,\\
e_1\ast e_1=\alpha e_1+\beta e_2, \ e_2\ast e_1=\gamma e_1, \ e_1\ast e_2=\gamma e_1, \ e_2\ast e_2=\gamma e_2.
\end{cases}
\]
Now we consider the automorphism: $\theta(e_1)=e_1, \ \theta(e_2)=ae_2$, where $a\neq0$.

By verifying all  multiplications of the algebra in the new basis we obtain the relations between the
parameters $\{\alpha', \beta', \gamma'\}$ and $\{\alpha, \beta, \gamma\}$:
\[
\begin{split}
\alpha'=\alpha,\quad
\beta'=\frac{\beta}{a},\quad
\gamma'=a\gamma.
\end{split}
\]
\begin{itemize}
    \item \textbf{Case $\gamma\neq0$.} By setting $a=\frac1{\gamma}$ we get $CA_{10}(\alpha, \beta)$.
\item \textbf{Case $\gamma=0$.} 
If $\beta\neq0$, then by setting $a=\beta$ we can assume $\beta'=1$. Hence, we have $CA_{11}(\alpha)$.
\end{itemize}

\textbf{Case 2.}
\[
\begin{cases}
e_1\circ e_1=e_1,\\
e_1\ast e_1=\alpha e_1+\beta e_2, \ e_1\ast e_2=\gamma e_2.
\end{cases}
\]
By considering the automorphism and verifying all  multiplications of the algebra in the new basis we write the relations between the
parameters $\{\alpha', \beta', \gamma'\}$ and $\{\alpha, \beta, \gamma\}$:
\[
\begin{split}
\alpha'=\alpha,\quad
\beta'=\frac{\beta}{a},\quad
\gamma'=\gamma.
\end{split}
\]
\begin{itemize}
    \item We can assume $\beta'=1$ without loss of generality if $\beta\neq0$ by setting $a=\beta$. Thus, we get $CA_{12}(\alpha, \gamma)$.
\end{itemize}

\textbf{Case 3.}
\[
\begin{cases}
e_1\circ e_1=e_1,\\
e_1\ast e_1=\alpha e_1, \ e_2\ast e_2=\beta e_2.
\end{cases}
\]
By considering the automorphism and verifying all  multiplications of the algebra in the new basis we write the relations between the
parameters $\{\alpha', \beta'\}$ and $\{\alpha, \beta\}$:
\[
\begin{split}
\alpha'=\alpha,\quad
\beta'=a\beta.
\end{split}
\]
\begin{itemize}
    \item If $\beta\neq0$, with choosing $a=\frac{1}{\beta}$, we can assume $\beta'=1$. Hence, we obtain $CA_{13}(\alpha)$.
\end{itemize}

3. \textbf{Construction from $A_3$.}

By Lemma \ref{cocycles} we have the following cases:

\textbf{Case 1.}
\[
\begin{cases}
e_1\circ e_1=e_2,\\
e_1\ast e_1=\alpha e_1+\beta e_2, \ e_1\ast e_2=\gamma e_2.
\end{cases}
\]
Now we consider the automorphism: $\theta(e_1)=ae_1+be_2, \ \theta(e_2)=a^2e_2$, where $a\neq0$.

By verifying all  multiplications of the algebra in the new basis we obtain the relations between the
parameters $\{\alpha', \beta', \gamma'\}$ and $\{\alpha, \beta, \gamma\}$:
\[
\begin{split}
\alpha'=a\alpha,\quad
\beta'=\frac{b}{a}(\gamma-\alpha)+\beta,\quad
\gamma'=a\gamma.
\end{split}
\]
\begin{itemize}
    \item \textbf{Case $\alpha=\gamma$.} If $\alpha\neq0$, we can assume $\alpha'=1$ without loss of generality.  Hence we get $CA_{14}(\beta)$.
\item \textbf{Case $\alpha\neq\gamma$.} Then by setting $b=\frac{a\beta}{\alpha-\gamma}$, we get $\beta'=0$. Further, if $\alpha=0$, by putting $a=\frac1{\gamma}$ we obtain $CA_{15}$.
Otherwise, by setting $a=\frac{1}{\alpha}$ we can obtain $CA_{16}(\gamma)$.
\end{itemize}

\textbf{Case 2.}

\[
\begin{cases}
e_1\circ e_1=e_2,\\
e_1\ast e_1=\alpha e_1+\beta e_2, \ e_2\ast e_1=\gamma e_1+\delta e_2.
\end{cases}
\]
By verifying all  multiplications of the algebra in the new basis we obtain the relations between the
parameters $\{\alpha', \beta', \gamma', \delta'\}$ and $\{\alpha, \beta, \gamma, \delta\}$:
\[
\begin{split}
\alpha'&=a\alpha+b\gamma,\\
\beta'&=\frac{b}{a}(\delta-\alpha)-\frac{b^2\gamma}{a^2}+\beta,\\
\gamma'&=a^3\gamma,\\
\delta'&=-ab\gamma+a\delta
\end{split}
\]
\begin{itemize}
\item \textbf{Case $\gamma\neq0$.}
By choosing $a=\frac{1}{\sqrt[3]{\gamma}}, \ b=-\frac{a\alpha}{\gamma}$ we can obtain $\alpha'=0, \gamma'=1$. Then we derive $CA_{17}(\beta, \delta)$.
\item \textbf{Case $\gamma=0$.}

If $\alpha\neq\delta$, by setting $b=\frac{a\beta}{\alpha-\delta}$ we can assume $\beta'=0$. Then we can obtain $CA_{18}( \delta)$ or $CA_{19}$.

\

If $\alpha=\delta$, then we can get $CA_{20}( \beta)$.
\end{itemize}
\textbf{Case 3.}
\[
\begin{cases}
e_1\circ e_1=e_2,\\
e_1\ast e_1=\alpha e_1+\beta e_2, \ e_2\ast e_1=\gamma e_2, \ e_1\ast e_2=\delta e_2.
\end{cases}
\]
By verifying all  multiplications of the algebra in the new basis we obtain the relations between the
parameters $\{\alpha', \beta', \gamma', \delta'\}$ and $\{\alpha, \beta, \gamma, \delta\}$:
\[
\begin{split}
\alpha'&=a\alpha,\quad
\beta'=\frac{b}{a}(\delta+\gamma-\alpha)+\beta,\\
\gamma'&=a\gamma,\quad
\delta'=a\delta.
\end{split}
\]
\begin{itemize}
    \item \textbf{Case $\alpha\neq\delta+\gamma$.}
By setting $b=\frac{\beta}{\alpha-\delta-\gamma}$, we get $\beta'=0$. Thus we have $CA_{21}(\alpha, \gamma, \delta)$.

\item \textbf{Case $\alpha=\delta+\gamma$.} If $\gamma=-\delta$, then by choosing $a$ we can obtain $CA_{22}(\beta)$.
Otherwise, $\alpha\neq0$ and by setting $a=\frac{1}{\alpha}$ we can get $CA_{23}(\beta, \gamma, \delta)$.
\end{itemize}

\textbf{Case 4.}
\[
\begin{cases}
e_1\circ e_1=e_2,\\
e_1\ast e_1=\alpha e_1+\beta e_2, \ e_2\ast e_1=\gamma e_1, \ e_1\ast e_2=\gamma e_1, \ e_2\ast e_2=\gamma e_2.
\end{cases}
\]
By verifying all  multiplications of the algebra in the new basis we obtain relations between the
parameters $\{\alpha', \beta', \gamma', \delta'\}$ and $\{\alpha, \beta, \gamma, \delta\}$:
\[
\begin{split}
\alpha'=a\alpha-2b\gamma,\quad 
\beta'=\frac{b}{a}\alpha-\frac{b^2}{a^2}\gamma+\beta,\quad
\gamma'=a^2\gamma.
\end{split}
\]
\begin{itemize}
\item \textbf{Case $\gamma=0$.} If $\alpha=0$, then we have $CA_{24}(\beta)$.
Otherwise, by setting $a=\frac{1}{\alpha}, b=-\frac{a\beta}{\alpha}$ we can obtain $CA_{25}$.
\item \textbf{Case $\gamma\neq0$.} By setting $a=\frac{1}{\sqrt{\gamma}}, \ b=\frac{a\alpha}{2\gamma}$ we can get $CA_{26}(\beta)$.
\end{itemize}

4. \textbf{Construction from $A_4$.}

By Lemma \ref{cocycles} we can write $CA_{27}(\alpha, \beta, \gamma,\delta)$. 
Now we consider the automorphism: $\theta(e_1)=e_2, \ \theta(e_2)=e_1$.

By verifying all  multiplications of the algebra in the new basis we obtain the relations between the
parameters $\{\alpha', \beta', \gamma', \delta'\}$ and $\{\alpha, \beta, \gamma, \delta\}$:
\[
\begin{split}
\alpha'=\delta,\quad 
\beta'=-\gamma,\quad 
\gamma'=-\beta,\quad 
\delta'=\alpha.
\end{split}
\]
This implies $CA_{27}(\alpha, \beta, \gamma,\delta)\cong CA_{27}(\delta, -\gamma, -\beta,\alpha)$.

5. \textbf{Construction from $A_5$.}
\
By Lemma \ref{cocycles} we have the following cases:

\textbf{Case 1.} We can write $CA_{28}(\alpha, \beta, \gamma)$.

Now we consider the automorphism: $\theta(e_1)=-e_1, \ \theta(e_2)=e_2$.

By verifying all  multiplications of the algebra in the new basis we obtain the relations between the
parameters $\{\alpha', \beta', \gamma'\}$ and $\{\alpha, \beta, \gamma\}$:
\[
\begin{split}
\alpha'=-\alpha,\quad 
\beta'=\beta,\quad
\gamma'=\gamma.
\end{split}
\]
This implies $CA_{28}(\alpha, \beta, \gamma)\cong CA_{28}(-\alpha, \beta, \gamma)$.

\textbf{Case 2.} We can write $CA_{29}(\alpha)$.
By verifying all  multiplications of the algebra in the new basis we obtain the relations between the
parameters $\{\alpha', \beta'\}$ and $\{\alpha, \beta\}$:
\[
\begin{split}
\alpha'=\alpha,\quad
\beta'=-\beta.
\end{split}
\]
This implies $CA_{29}(\alpha, \beta)\cong CA_{29}(\alpha, -\beta)$.

6. \textbf{Construction from $A_6(\lambda)$.} By Lemma \ref{cocycles} we have the following cases:

\textbf{Case $\lambda=0$.}

\textbf{Case 1.}
\[
\begin{cases}
e_2\circ e_1=-e_1, \\
e_1\ast e_1=-\alpha e_1+\beta e_2, \ e_2\ast e_1=\gamma e_1+\alpha e_2.
\end{cases}
\]
Now we consider the automorphism: $\theta(e_1)=ae_1, \ \theta(e_2)=e_2$, where $a\neq0$.

By verifying all  multiplications of the algebra in the new basis we obtain the relations between the
parameters $\{\alpha', \beta', \gamma'\}$ and $\{\alpha, \beta, \gamma\}$:
\[
\begin{split}
\alpha'=a\alpha,\quad 
\beta'=a^2\beta,\quad 
\gamma'=\gamma.
\end{split}
\]
\begin{itemize}
    \item \textbf{Case $\alpha\neq0$.} By setting $a=\frac1{\alpha}$, we get $CA_{30}(\beta, \gamma)$.

\item \textbf{Case $\alpha=0$.} If $\beta\neq0$, by setting $a=\frac{1}{\sqrt{\beta}}$ we can obtain $\beta'=1$. Hence,  we have $CA_{31}(\gamma)$.
\end{itemize}

\textbf{Case 2.} We can write $CA_{32}(\alpha,\beta)$.
By verifying all  multiplications of the algebra in the new basis we obtain the relations between the
parameters $\{\alpha', \beta'\}$ and $\{\alpha, \beta\}$:
\[
\begin{split}
\alpha'=\alpha,\quad
\beta'=\beta.
\end{split}
\]
This implies that the algebra is isomorphic to itself.

\textbf{Case 3.}
\[
\begin{cases}
e_2\circ e_1=-e_1, \\
e_2\ast e_1=\alpha e_1, \ e_2\ast e_2=\beta  e_1+\gamma e_2.
\end{cases}
\]
By verifying all  multiplications of the algebra in the new basis we obtain the relations between the
parameters $\{\alpha', \beta',\gamma'\}$ and $\{\alpha,\beta,\gamma\}$:
\[
\begin{split}
\alpha'=\alpha,\quad 
\beta'=\frac{\beta}{a},\quad 
\gamma'=\gamma.
\end{split}
\]
\begin{itemize}
\item  If $\beta\neq0$, by setting $a=\beta$, we get $CA_{33}(\alpha,\gamma)$.
\end{itemize}

\textbf{Case $\lambda=-1$.}
\[
\begin{cases}
e_2\circ e_1=-e_1, \ e_2\circ e_2=-e_2 \\
e_1\ast e_1=\alpha e_1, \ e_2\ast e_1=\beta e_1, \ e_1\ast e_2=\alpha e_2, \ e_2\ast e_2=\beta e_2.
\end{cases}
\]
By verifying all  multiplications of the algebra in the new basis we obtain the relations between the
parameters $\{\alpha',\beta'\}$ and $\{\alpha,\beta\}$:
\[
\begin{split}
\alpha'=a\alpha,\quad 
\beta'=\beta.
\end{split}
\]
\begin{itemize}
    \item If $\alpha\neq0$, by setting $a=\frac{1}{\alpha}$ we derive $\alpha'=1$. Hence, we have $CA_{34}(\beta)$.
\end{itemize}

\textbf{Case $\lambda\neq\{-1,0\}$.}
\[
\begin{cases}
e_2\circ e_1=-e_1, \ e_2\circ e_2=\lambda e_2 \\
e_2\ast e_1=\alpha e_1, \ e_1\ast e_2=\beta e_1+\gamma e_2.
\end{cases}
\]
By verifying all  multiplications of the algebra in the new basis we obtain the relations between the
parameters $\{\alpha',\beta', \gamma'\}$ and $\{\alpha,\beta,\gamma\}$:
\[
\begin{split}
\alpha'=\alpha,\quad 
\beta'=\beta,\quad 
\gamma'=a\gamma.
\end{split}
\]
\begin{itemize}

\item  If $\gamma\neq0$, by setting $a=\frac{1}{\gamma}$, we obtain $\gamma'=1$. Hence we have $CA_{35}(\alpha, \beta)$.
\end{itemize}

7. \textbf{Construction from $A_7$.} By Lemma \ref{cocycles} we have :
\[
\begin{cases}
e_2\circ e_1=-e_1, \ e_2\circ e_2=e_1-e_2 \\
e_2\ast e_1=\alpha e_1, \ e_2\ast e_2=\beta e_1+\gamma e_2.
\end{cases}
\]
Now we consider the automorphism: $\theta(e_1)=e_1, \ \theta(e_2)=ae_1+e_2$.

By verifying all  multiplications of the algebra in the new basis we obtain the relations between the
parameters $\{\alpha', \beta', \gamma'\}$ and $\{\alpha, \beta, \gamma\}$:
\[
\begin{split}
\alpha'=\alpha,\quad
\beta'=\beta+a(\alpha-\gamma),\quad
\gamma'=\gamma.
\end{split}
\]
\begin{itemize}
\item \textbf{Case $\alpha\neq\gamma$.}
By setting $a=\frac{\beta}{\gamma-\alpha}$ we obtain $CA_{36}(\alpha, \gamma)$.

\item \textbf{Case $\alpha=\gamma$.} Then, we get $CA_{37}(\alpha, \beta)$.
\end{itemize}

8. \textbf{Construction from $A_8(\lambda)$.} By Lemma \ref{cocycles} we have the following cases:

\textbf{Case $\lambda\neq\{-2,0\}$.}
\[
\begin{cases}
e_1\circ e_2=(\lambda+1)e_1, \ e_2\circ e_1=\lambda e_1, \ e_2\circ e_2=(\lambda-1)e_2 \\
e_2\ast e_1=(\alpha+\beta) e_1, \ e_1\ast e_2=2\alpha e_1, \ e_2\ast e_2=\gamma e_1+2\beta e_2.
\end{cases}
\]
Now we consider the automorphism: $\theta(e_1)=ae_1, \ \theta(e_2)=e_2$, where $a\neq0$.

By verifying all  multiplications of the algebra in the new basis we obtain the relations between the
parameters $\{\alpha', \beta', \gamma'\}$ and $\{\alpha, \beta, \gamma\}$:
\[
\begin{split}
\alpha'=\alpha,\quad 
\beta'=\beta,\quad 
\gamma'=\frac{\gamma}{a}.
\end{split}
\]
\begin{itemize}
    \item We can assume $\gamma'=1$ by setting $a=\gamma$ when $\gamma\neq0$. Hence, we obtain $CA_{38}(\alpha, \beta)$.
\end{itemize}

\textbf{Case $\lambda=-2$.}
\[
\begin{cases}
e_1\circ e_2=-e_1, \ e_2\circ e_1=-2e_1, \ e_2\circ e_2=-3e_2 \\
e_1\ast e_1=3\alpha e_1,\ e_2\ast e_1=2\beta e_1+\alpha e_2, \ e_1\ast e_2=\gamma e_1+2\alpha e_2, \ e_2\ast e_2=3\beta e_2.
\end{cases}
\]
By verifying all  multiplications of the algebra in the new basis we obtain the relations between the
parameters $\{\alpha', \beta', \gamma'\}$ and $\{\alpha, \beta, \gamma\}$:
\[
\begin{split}
\alpha'=a\alpha,\quad 
\beta'=\beta,\quad 
\gamma'=\gamma.
\end{split}
\]
\begin{itemize}
    \item If $\alpha\neq0$, by setting $a=\frac{1}{\alpha}$ we can assume $\alpha'=1$. Thus we get $CA_{39}(\beta,\gamma)$.
\end{itemize}

\textbf{Case $\lambda=0$.}

\textbf{Case 1.}
\[
\begin{cases}
e_1\circ e_2=e_1, \  e_2\circ e_2=-e_2 \\ e_2\ast e_1=(\alpha +\beta) e_1, \ e_1\ast e_2=2\alpha e_1, \ e_2\ast e_2=\gamma e_1+2\beta e_2.
\end{cases}
\]
By verifying all  multiplications of the algebra in the new basis we obtain the relations between the
parameters $\{\alpha', \beta', \gamma'\}$ and $\{\alpha, \beta, \gamma\}$:
\[
\begin{split}
\alpha'=\alpha,\quad 
\beta'=\beta,\quad
\gamma'=\frac{\gamma}a.
\end{split}
\]
\begin{itemize}
    \item If $\gamma\neq0$, by setting $a=\gamma$ we  can assume $\gamma'=1$. Hence, we obtain $CA_{40}(\alpha,\beta)$.
\end{itemize}

\textbf{Case 2.}
\[
\begin{cases}
e_1\circ e_2=e_1, \  e_2\circ e_2=-e_2 \\ e_1\ast e_2=\alpha e_1+\beta e_2, \ e_2\ast e_2=\gamma e_1+\delta e_2.
\end{cases}
\]
By verifying all  multiplications of the algebra in the new basis we obtain the relations between the
parameters $\{\alpha', \beta', \gamma',\delta'\}$ and $\{\alpha, \beta, \gamma,\delta\}$:
\[
\begin{split}
\alpha'=\alpha,\quad 
\beta'=a\beta,\quad 
\gamma'=\frac{\gamma}a,\quad 
\delta'=\delta.
\end{split}
\]
\begin{itemize}
    \item \textbf{Case $\gamma\neq0$.} By setting $a=\gamma$ we get $\gamma'=1$. Thus, we derive $CA_{41}(\alpha, \beta, \gamma)$.
\item \textbf{Case $\gamma=0$.} If $\beta\neq0$ by setting $a=\frac1{\beta}$ we get $\beta'=1$. Thus, we have $CA_{42}(\alpha, \beta)$.
\end{itemize}

\textbf{Case 3.}
\[
\begin{cases}
e_1\circ e_2=e_1, \  e_2\circ e_2=-e_2 \\ e_1\ast e_1=-\alpha e_1, \ e_2\ast e_1=\alpha e_2, \ e_1\ast e_2=-\beta e_1, \ e_2\ast e_2=\beta e_2.
\end{cases}
\]
By verifying all  multiplications of the algebra in the new basis we obtain the relations between the
parameters $\{\alpha', \beta'\}$ and $\{\alpha, \beta\}$:
\[
\begin{split}
\alpha'=a\alpha,\quad 
\beta'=\beta.
\end{split}
\]
\begin{itemize}
    \item When $\alpha\neq0$ we can assume $\alpha'=1$ by setting $a=\frac{1}{\alpha}$. Thus, we get  $CA_{43}(\beta)$.
\end{itemize}

9. \textbf{Construction from $A_9$.} By Lemma \ref{cocycles} we can write
\[
\begin{cases}
e_1\circ e_2=-e_1, \ e_2\circ e_1=-2 e_1, \ e_2\circ e_2=e_1-3e_2 \\
e_2\ast e_1=(\alpha+\beta) e_1, \ e_1\ast e_2=2\alpha e_1, \ e_2\ast e_2=\gamma e_1+2\beta e_2.
\end{cases}
\]
Now we consider the automorphism: $\theta(e_1)=e_1, \ \theta(e_2)=ae_1+e_2$.

By verifying all  multiplications of the algebra in the new basis we obtain the relations between the
parameters $\{\alpha', \beta', \gamma'\}$ and $\{\alpha, \beta, \gamma\}$:
\[
\begin{split}
\alpha'=\alpha,\quad 
\beta'=\beta,\quad 
\gamma'=\gamma+a(3\alpha-\beta).
\end{split}
\]
\begin{itemize}
    \item \textbf{Case  $3\alpha-\beta\neq0$.}
We may assume $\gamma'=0$ without loss of generality by setting $a=-\frac{\gamma}{3\alpha-\beta}$. Hence, we have $CA_{44}(\alpha, \beta)$.
\item \textbf{Case 2 $3\alpha-\beta=0$.} Then, we have $CA_{45}(\alpha, \gamma)$.
\end{itemize}

\end{proof}

\section*{Acknowledgments}
The author would like to express his gratitude to Yunhe Sheng for many fruitful
discussions and suggestions in the preparation of this article.

\end{document}